\documentclass[leqno, 11pt]{article}
\usepackage[utf8]{inputenc}

\usepackage{tikz-cd}

\usepackage{amsmath,amsthm,amscd,amssymb,amsbsy,amstext,mathtools,cite}
\usepackage{pdfpages}
\usepackage{enumerate}
\usepackage{mathrsfs}
\usepackage[utf8]{inputenc}
\usepackage{geometry}[margin=1in]
\numberwithin{equation}{section}
\usepackage{amsfonts}
\usepackage[T1]{fontenc}
\usepackage{graphicx}
\usepackage{subcaption}
\usepackage{hhline}
\usepackage{placeins}
\usepackage{ulem}

\usepackage{hyperref}
\usepackage{url}

\newcommand{\eqindent}{\displayindent0pt\displaywidth\textwidth}

\newtheorem{theorem}{Theorem}[section]
\newtheorem{proposition}[theorem]{Proposition}
\newtheorem{lemma}[theorem]{Lemma}
\newtheorem{corollary}[theorem]{Corollary}

\theoremstyle{definition}
\newtheorem{definition}[theorem]{Definition}
\newtheorem{example}[theorem]{Example}

\newtheorem{convention}[theorem]{Convention}

\theoremstyle{remark}
\newtheorem{remark}[theorem]{Remark}

\newcommand{\norm}[1]{\|#1\|}

\newcommand{\set}[1]{\left\{#1\right\}}

\newcommand{\abs}[1]{\left|#1\right|}

\newcommand{\R}{\mathbb{R}}

\renewcommand{\epsilon}{\varepsilon}
\newcommand{\eps}{\epsilon}

\newcommand{\HD}{\mathcal{D}}

\newcommand{\define}[1]{\textbf{#1}}

\newcommand{\brac}[1]{\left(#1\right)}

\renewcommand{\P}{\mathcal{P}}
\newcommand{\M}{\mathcal{M}}

\newcommand{\C}{\mathcal{C}} 
\newcommand{\Cs}{\mathcal{S}} 

\newcommand{\w}{\mathsf{W}}
\newcommand{\gw}{\mathsf{GW}}
\newcommand{\pgw}{{\mathsf{PGW}}}
\newcommand{\pgwt}{\mathsf{mPGW}}
\newcommand{\spgw}{{\mathsf{sPGW}}}
\newcommand{\dis}{\mathsf{dis}}

\renewcommand{\sharp}{\#}

\definecolor{blackscomment}{RGB}{192,23,118}

\definecolor{kaiyingscomment}{RGB}{194,94,64}

\definecolor{tomscomment}{RGB}{25,15,150}

\title{Metric properties of partial and robust Gromov-Wasserstein distances}

\usepackage[affil-it]{authblk}

\author[1]{Jannatul Chhoa}
\author[2]{Michael Ivanitskiy}
\author[3]{Fushuai Jiang}
\author[4]{Shiying Li}
\author[5]{Daniel McBride}
\author[6]{Tom Needham}
\author[7]{Kaiying O'Hare}

\affil[1]{University of Houston}
\affil[2]{Colorado School of Mines}
\affil[3]{University of Maryland}
\affil[4]{University of Nebraska - Lincoln}
\affil[5]{University of Tennessee - Knoxville}
\affil[6]{Florida State University}
\affil[7]{Johns Hopkins University}

\date{{\small
}}

\begin{document}
\maketitle

\begin{abstract}

The Gromov-Wasserstein (GW) distances define a family of metrics, based on ideas from optimal transport, which enable comparisons between probability measures defined on distinct metric spaces. They are particularly useful in areas such as network analysis and geometry processing, as computation of a GW distance involves solving for registration between the objects which minimizes geometric distortion. Although GW distances have proven useful for various applications in the recent machine learning literature, it has been observed that they are inherently sensitive to outlier noise and cannot accommodate partial matching. This has been addressed by various constructions building on the GW framework; in this article, we focus specifically on a natural relaxation of the GW optimization problem, introduced by Chapel et al., which is aimed at addressing exactly these shortcomings. Our goal is to understand the theoretical properties of this relaxed optimization problem, from the viewpoint of metric geometry. While the relaxed problem fails to induce a metric, we derive precise characterizations of how it fails the axioms of non-degeneracy and triangle inequality. These observations lead us to define a novel family of distances, whose construction is inspired by the Prokhorov and Ky Fan distances, as well as by the recent work of Raghvendra et al.\ on robust versions of classical Wasserstein distance. We show that our new distances define true metrics, that they induce the same topology as the GW distances, and that they enjoy additional robustness to perturbations. These results provide a mathematically rigorous basis for using our robust partial GW distances in applications where outliers and partial matching are concerns.
\end{abstract}

\section{Introduction}

Registration is an important preprocessing task in fields such as machine learning, network analysis, and geometry processing, which broadly consists of finding a correspondence between a pair of objects (images, point clouds, graphs, etc.) that optimally aligns the objects' respective geometrical features. This general process was mathematically formalized in the language of metric geometry and optimal transport in the seminal work of M\'emoli \cite{memoli2007,memoli2011-gw}, who introduced the \textit{Gromov-Wasserstein (GW) distances}. In these works, objects to be registered are flexibly modeled as \textit{metric measure spaces} (\textit{mm-spaces}), or triples of the form $(X,d_X,\mu_X)$, where $(X,d_X)$ is a compact metric space and, in this original formulation, $\mu_X$ is a Borel probability measure on $X$. For $p \geq 1$, the \textit{Gromov-Wasserstein $p$-distance} between a pair of mm-spaces is given by the quantity\footnote{In the original papers \cite{memoli2007,memoli2011-gw}, the metric was defined to be $\frac{1}{2}$ of this quantity, since this aligns it with Gromov-Hausdorff distance in a certain sense. We forgo this factor in this paper for the sake of simpler analysis.}
\begin{equation}\label{eq:gw_intro}
\inf_\pi \left(\int_{X \times Y} \int_{X \times Y} |d_X(x,x') - d_Y(y,y')|^p d\pi(x,y) d\pi(x',y') \right)^{1/p},
\end{equation}
where the infimum is taken over \textit{couplings} $\pi$, or joint probability measures on $X \times Y$ whose marginals are $\mu_X$ and $\mu_Y$, respectively. This is shown in \cite{memoli2007} to yield a metric on the space of isomorphism classes of mm-spaces. An important feature of this metric is that its evaluation involves computing a certain notion of registration of the objects $X$ and $Y$; that is, a coupling $\pi$ minimizing \eqref{eq:gw_intro} can be viewed as a `soft correspondence' between $X$ and $Y$, with the quantity $d\pi(x,y)$ representing the amount of mass from $x \in X$ which should be registered to $y \in Y$. Due to the structure of \eqref{eq:gw_intro}, the optimality of such a correspondence is determined by the extent to which it distorts the metric structure of the spaces, on average. 

Due to the ubiquity of the registration problem, and the flexibility of this framework for addressing it, Gromov-Wasserstein distances have become an important tool for applications---see, e.g., \cite{peyre2016gromov,chowdhury2020gromov,xu2019gromov,chowdhury2021generalized,demetci2022scot,govek2023cajal}. However, a drawback of this approach to registration, and to optimal transport-based methods in general, is that it is susceptible to outlier noise and is unable to handle partial matching problems due to the nature of its formulation: couplings necessarily register all of the mass of $X$ to points in $Y$ and vice-versa. To address these potential shortcomings, various notions of \textit{unbalanced}, \textit{partial} or \textit{robust} Gromov-Wasserstein distance have been introduced in the recent literature \cite{chapel2020partial,sejourne2021-conic,Vanderbilt2024-PGW,kong2024outlier}.

This paper was inspired by the rather natural partial Gromov-Wasserstein problem introduced by Chapel et al.\ in \cite{chapel2020partial}, where the objective function of the optimization problem \eqref{eq:gw_intro} remains the same, but the feasible set is expanded to allow joint measures whose marginals may differ from the targets $\mu_X$ and $\mu_Y$ in a controlled manner (see Section \ref{sect: relation to chapel} for a precise formulation). The partial GW problem is successfully applied to positive-unlabeled learning tasks in \cite{chapel2020partial}, but our interest is in its theoretical properties: namely, does the solution of the partial GW optimization problem define a metric in the same way that the solution of the original GW problem does? Perhaps unsurprisingly, the answer is `no', but we present more refined results which shed light on exactly how it fails. Using this structure as a starting point, we define a new distance, which we refer to as the \textit{robust partial Gromov-Wasserstein distance}. We show that robust partial GW distance does define a metric, which induces the same topology as the original GW distance, but which has improved resiliency to noise. Let us now state more precisely our main contributions.

\paragraph{Main Contributions and Outline.} In Section \ref{sec:basics}, we generalize the partial GW problem of \cite{chapel2020partial} (Definition \ref{def:Chapel}) (it was initially defined only for finite spaces and for the  $p=2$ objective function, and we extend it in the obvious way) and define a reformulation that is more natural to work within most of our proofs---the new formulation is in Definition \ref{def:partial_GW} and we establish its equivalence to that of Chapel et al.\ in Proposition \ref{prop:equivalence} and Corollary \ref{cor:equivalence}. When applied to a pair of mm-spaces, the reformulation is denoted $\pgw_{\eps,p}(X,Y)$, where we abuse notation and write $X = (X,d_X,\mu_X)$ and $Y = (Y,d_Y,\mu_Y)$ for mm-spaces, and where the $\eps$ is a hyperparameter controlling the extent to which joint measures can deviate from being couplings of the target measures $\mu_X$ and $\mu_Y$. Furthermore, we introduce a novel variant of this idea in Definition \ref{def:partial_GW}, which gives a more symmetric control over this deviation; this symmetrized version is denoted $\spgw_{\eps,p}$. Theoretical contributions in Section \ref{sec:basics} include a proof that solutions to the optimization problems associated with $\pgw_{\eps,p}$ and $\spgw_{\eps,p}$ are always realized (Theorem \ref{thm:realization}) and a result which states that both $\pgw_{\eps,p}$ and $\spgw_{\eps,p}$ are right-continuous (Theorem \ref{thm:monotone-convergence}) in the hyperparameter $\eps$. In particular, 
as $\eps$ tends to zero, $\pgw_{\eps,p}$ and $\spgw_{\eps,p}$ converge to the classical Gromov-Wasserstein distance (Theorem \ref{thm:GW-convergence}).

The metric properties of $\pgw_{\eps,p}$ and $\spgw_{\eps,p}$ are studied in detail in Section \ref{sec:metric_properties_spgw}. In Corollary \ref{cor: PGW equal 0}, we give a precise characterization of the relationship satisfied by $X$ and $Y$ when $\pgw_{\eps,p}(X,Y)$ or $\spgw_{\eps,p}(X,Y)$ vanish. This follows from Theorem \ref{thm:approxi-nondegen}, which expresses $\pgw_{\eps,p}(X,Y)$ and $\spgw_{\eps,p}(X,Y)$ in terms of classical Gromov-Wasserstein distances between spaces whose masses have been redistributed in a controlled fashion. We show by example that the triangle inequality fails to hold for $\pgw_{\eps,p}$ and $\spgw_{\eps,p}$, but that $\spgw_{\eps,p}$ does satisfy a certain technical, approximate version of the triangle inequality (Theorem \ref{thm:relaxed triangle}). The proof of the approximate triangle inequality requires Lemma \ref{lem:gluing-Linfinity}, which is a generalization of the famous Gluing Lemma from optimal transport theory \cite[Chapter 6]{villani-ot-book}. Using some of these results, we establish in Section \ref{sect: related works} precise connections to other relaxed versions of GW distance that have appeared in the recent literature. 

Using our formulation of $\pgw_{\eps,p}$ as a starting point, we employ in Section \ref{sec:robust_pGW} a construction similar to that of \cite{RSZ2024} (also reminiscent of that of the Prokhorov distance) to define a new distance that is independent of any relaxation parameter $\eps$, which we refer to as \textit{robust partial Gromov-Wasserstein distance}. We show in Corollary \ref{cor:pgwpkmetric} that the robust partial GW distance is a metric on the space of isomorphism classes of compact mm-spaces. This metric is shown to be incomplete for the space of compact \textit{probability metric measure spaces} (\textit{pmm-spaces}) (which are some restricted mm-spaces) in Theorem \ref{thm:incomplete}, to be topologically equivalent to the classical Gromov-Wasserstein distance in Theorem \ref{thm:equivalence_of_topologies}, and to satisfy a robustness property with respect to perturbations in the input data in Corollary \ref{cor:robust}. 

The contributions of this paper focus on the theoretical properties of the various distances. However, the need for partial and noise-robust registration methods in practical data analysis tasks motivates future applications-oriented work on our robust partial GW framework. In particular, follow-up work will focus on stable and efficient numerical implementations of our constructions. To illustrate the potential practical benefits of our framework, we provide some proof-of-concept numerical results in Section \ref{sec:numerics}.

\section{Basics of Partial Gromov-Wasserstein Distances}\label{sec:basics}

This section introduces the main concepts that will be used throughout the paper.

\subsection{Measure Theory Background}

We begin with some background definitions and notation from measure theory. The concepts from this subsection are standard, besides, perhaps, the notions of relaxed couplings defined below.

\paragraph{Measures and derivatives.}

Let $X$ be a measure space. We denote the set of nonnegative measures on $X$ by $\mathcal{M}_+(X)$. Given $\mu \in \mathcal{M}_+(X)$, we use $\abs{\mu}$ to denote the total mass (or equivalently, the total variation) of $\mu$; that is, $\abs{\mu} = \mu(X) = \int_X d\mu$. The set of probability measures on $X$, i.e., $\mu \in \mathcal{M}_+(X)$ with $\abs{\mu}=1$, will be denoted $\mathcal{P}(X)$. 

Given $\mu,\nu \in \mathcal{M}_+(X)$, we use $\frac{d\mu}{d\nu}$ to denote the Radon-Nikodym derivative of $\mu$ with respect to $\nu$. We are interested in the case $\frac{d\mu}{d\nu} \in L^\infty(\nu)$, which is equivalent to saying that $\mu$ can be obtained from $\nu$ by conditioning \cite{Diaconis1982-conditional}. 

We note that two Borel probability measures $\mu$ and $\nu$ satisfy $\norm{\frac{d\mu}{d\nu}}_{L^\infty(\nu)} \leq r \in (1,\infty)$ if and only if $  \mu(A)\leq r\cdot \nu(A)$ for every Borel set $A$, which is the same as $\norm{\frac{d\mu}{d\nu} - 1}_{L^\infty(\nu)}\leq r-1$.

\paragraph{Metric measure spaces.} 

A \define{metric measure space} (\define{mm-space}) \cite{gromov1999metric} is a tuple $(X,d_X,\mu_X)$, where $(X,d_X)$ is a complete and separable metric space and $\mu_X \in \mathcal{M}_+(X)$ (where $X$ is considered as a measure space with respect to its Borel $\sigma$-algebra). When the context is clear, we abuse notation and write $X$ instead of $(X, d_X, \mu_X)$. 

A \define{probability metric measure space} (\define{pmm-space}) is an mm-space with $\mu_X \in \P(X)$. The space $\P(X)$ can be endowed with the usual weak topology (sometimes also called the narrow topology in the literature): a sequence $(\mu_k)\subset\P(X)$ converges to $\mu \in \P(X)$ weakly if $\int_X fd\mu_k \to \int_X fd\mu$ for all bounded continuous functions $f : X \to \R$. For pmm-spaces, Prokhorov's Theorem states that the topology of weak convergence is induced by a metric, and $\P(X)$ is precompact in this topology.  

We say that an mm-space $X$ is \define{compact} if its underlying metric space $(X,d_X)$ is compact. An mm-space is called \define{fully-supported} if the support of its measure $\mu_X$ is all of $X$. We say two mm-spaces $X$ and $Y$ are \define{isomorphic} if there exists a measure-preserving isometric bijection $f:X \to Y$. 

\begin{remark}\label{rem:compact_fully_supported}
    Our results frequently assume that mm- and pmm-spaces are compact and fully supported for the sake of convenience; many results could be generalized at the expense of requiring additional technical qualifiers and notation.
\end{remark}

\paragraph{Couplings and relaxed couplings.}

Let $(X,d_X,\mu_X),(Y,d_Y,\mu_Y)$ be pmm-spaces. We use $p^X:X \times Y \to X$ and $p^Y:X \times Y \to Y$ to denote the canonical projections. Given a measure $\pi \in \mathcal{M}_+(X \times Y)$, we write $\pi_X = p^X_\sharp \pi$ and $\pi_Y = p^Y_\sharp \pi$, where the subscript $\sharp$ denotes the pushforward of a map, i.e., $f_\sharp \mu(S) = \mu(f^{-1}(S))$. We recall that the family of \define{(exact) couplings} of $\mu_X$ and $\mu_Y$ is defined to be
\begin{equation}\label{eqn:couplings}
    \C(\mu_X,\mu_Y) := \set{\pi \in \mathcal{P}(X\times Y) : \pi_X = \mu_X,\, \pi_Y = \mu_Y}.
\end{equation}

With a view toward introducing relaxed notions of optimal transport-based distances, we now define two additional types of couplings, both given by relaxing the equality in \eqref{eqn:couplings}. 

\begin{definition}[Relaxed Couplings]\label{def:relaxed_couplings}
Given $\eps = (\eps_1,\eps_2) \in [0,\infty]^2$ and a pair of probability spaces
$(X,\mu_X),(Y,\mu_Y)$, we define the set of \define{$\epsilon$-relaxed couplings} to be
\begin{equation}
    \C_\eps(\mu_X,\mu_Y) := \set{
    \pi \in \P(X\times Y) : 
    \pi_X \leq (1+\eps_1)\mu_X
    \text{ and }
    \pi_Y \leq (1+\eps_2)\mu_Y
    },
\end{equation}
where the inequality $\pi_X \leq (1+\eps_1)\mu_X$ means that $\pi_X(A) \leq (1+\eps_1)\mu_X(A)$ for all Borel measurable sets $A \subset X$, and likewise for the other inequality. Similarly, we define the set of \define{symmetric $\epsilon$-relaxed couplings} to be
\begin{equation}
    \Cs_{\eps}(\mu_X,\mu_Y) := \set{\pi \in \P(X\times Y)
    :
    \begin{matrix}
    \norm{\frac{d\pi_X}{d\mu_X} }_{L^\infty(\mu_X)},\, \norm{\frac{d\mu_X}{d\pi_X}}_{L^\infty(\pi_X)} \leq 1+\eps_1 ;\\
    \norm{\frac{d\pi_Y}{d\mu_Y} }_{L^\infty(\mu_Y)},\, \norm{\frac{d\mu_Y}{d\pi_Y} }_{L^\infty(\pi_Y)} \leq 1+\eps_2
    \end{matrix}
    }.
    \label{eq:def:eps-coupling}
\end{equation}
\end{definition}
The intuition behind these definitions is that the spaces of $\epsilon$-relaxed couplings allow measures whose marginals are allowed to redistribute mass in a controlled fashion.  

\begin{remark}
    The condition $\norm{\frac{d\mu_X}{d\pi_X}}_{L^\infty(\pi_X)} \leq 1+\eps_1$, in particular, implies that $\pi_X \leq (1+\eps_1)\mu_X$, in the sense described above. It follows that $\mathcal{S}_\eps(\mu_X,\mu_Y) \subset \mathcal{C}_\eps(\mu_X,\mu_Y)$. 
\end{remark}

\begin{remark}\label{rmk:relaxed_at_zero}
Using the shorthand $0=(0,0)$, observe that
\begin{equation}
    \C_{0}(\mu_X,\mu_Y) = \C(\mu_X,\mu_Y) = \Cs_0(\mu_X,\mu_Y).
\end{equation}
Namely, when no mass redistribution is allowed in the marginals, the `relaxed' couplings become exact. 
\end{remark}

\subsection{Partial Gromov-Wasserstein Distances}

We now introduce relaxed versions of the Gromov-Wasserstein distance and derive some basic properties. First, we recall some properties of the Gromov-Wasserstein distance itself.

\paragraph{Gromov-Wasserstein distance.} Let $X$ and $Y$ be pmm-spaces and let $p \in [1,\infty]$. We consider $d_X - d_Y$ as a function on $X \times Y \times X \times Y$, so that the quantity $\norm{d_X - d_Y}_{L^p(\pi\otimes \pi)}$ is well-defined for any measure $\pi \in \mathcal{M}_+(X \times Y)$ (in general, e.g., without a compactness assumption, it may be equal to $\infty$). Explicitly, for $p< \infty$, this is given by  
\begin{equation}
 \norm{d_X - d_Y}_{L^p(\pi\otimes \pi)}=   \brac{\int_{X\times Y}\int_{X\times Y}\abs{d_X(x,x') - d_Y(y,y')}^pd\pi(x,y)d\pi(x',y')}^{1/p}.
\end{equation}

With the above notation, the \define{Gromov-Wasserstein (GW) $p$-distance} between $X$ and $Y$ is given by 
\begin{equation}\label{eqn:gromov-wasserstein}
\gw_p(X,Y) \coloneqq \inf_{\pi \in \C(\mu_X,\mu_Y)}\norm{d_X - d_Y}_{L^p(\pi\otimes \pi)}.
\end{equation}
This notion of distance was introduced by M\'{e}moli in~\cite{memoli2007}, with additional properties established in~\cite{memoli2011-gw}. Particularly important properties are described in the following theorem, which combines results from \cite{memoli2007,memoli2011-gw}.

\begin{theorem}[\hspace{1sp}\cite{memoli2007,memoli2011-gw}]\label{thm:GW_properties}
\begin{enumerate}[(A)]
    \item For compact pmm-spaces $X$ and $Y$, the infimum in \eqref{eqn:gromov-wasserstein} is always realized by a coupling $\pi \in \C(\mu_X,\mu_Y)$ \cite[Corollary 10.1]{memoli2011-gw}. 
    \item Gromov-Wasserstein distance $\gw_p$ defines a pseudometric on the space of compact, fully-supported pmm-spaces, with $\gw_p(X,Y) = 0$ if and only if $X$ and $Y$ are isomorphic \cite[Theorem 5.1(a)]{memoli2011-gw}. 
\end{enumerate}
\end{theorem}

\begin{remark}
    Some of the conditions in these results can be weakened---see \cite{chowdhury2019gromov} or \cite{sturm2023space}, where some generalizations are established, or \cite{memoli2024comparison} which surveys properties of the Gromov-Wasserstein distance under various assumptions. As was mentioned in Remark \ref{rem:compact_fully_supported}, we frequently work at the level of generality of~\cite{memoli2011-gw}---that is, compact and fully supported spaces.
\end{remark}

Gromov-Wasserstein distances (and variants thereof---see, e.g., \cite{vayer2020fused,titouan2020co,zhang2024geometry,bauer2024z}) have been widely used in recent data-oriented applications, such as shape analysis \cite{peyre2016gromov,chowdhury2020gromov}, network science \cite{xu2019gromov,chowdhury2021generalized} and computational biology \cite{demetci2022scot,govek2023cajal}. A key feature of the GW framework is that the computation of the distance itself leads to useful information---a coupling minimizing \eqref{eqn:gromov-wasserstein} can be understood as defining a `soft correspondence' between the pmm-spaces. This correspondence can be used to register the points between the spaces, which is crucial for downstream analysis tasks.

The terminology \textit{Gromov-Wasserstein} is intended to signal that the GW distances combine ideas used to define two well-known, classical distances:
\begin{itemize}
    \item The \textit{Gromov-Hausdorff distance} is a metric on the space of isometry classes of compact metric spaces (which have not been endowed with a measure)---see~\cite{burago2001course} for a thorough discussion of its properties. Roughly, its formulation `lifts' the idea of the Hausdorff distance between subsets of a common metric space to the setting of distinct spaces.
    \item The \textit{Wasserstein distance} from classical optimal transport theory (see~\cite{villani-ot-book}) defines a metric on the space of probability distributions over a common metric space. We now recall the details, as they will be useful later in the paper. Let $(M,d)$ be a metric space and let $\mu,\nu$ be two probability measures on $M$. The \define{Wasserstein $p$-distance} between $\mu$ and $\nu$ is defined by 
\begin{equation*}
    \w_p(\mu,\nu) \coloneqq \inf_{\pi \in \mathcal{C}(\mu,\nu)} \brac{\int_{M}d(m,m')^pd\pi(m,m')}^{1/p} = \inf_{\pi \in \mathcal{C}(\mu,\nu)} \|d\|_{L^p(\pi)},
\end{equation*}
for $p \in [1,\infty)$, with the definition extending naturally to the $p=\infty$ case.
\end{itemize}
In reference to the ideas described above, the \textit{Gromov-Wasserstein distance} can be seen as a `Gromov-ization' of the Wasserstein distance, in the sense that it lifts Wasserstein distances to the setting where measures on distinct metric spaces can be compared. Given two pmm-spaces $X = (M,d,\mu)$ and $Y = (M,d,\nu)$ having the same underlying metric space $(M,d)$, we always have
\begin{equation}
    \gw_p(X,Y) \leq 2 \cdot \w_p(\mu,\nu)
    \label{eq:GWp-Wp}
\end{equation}
(see~\cite[Theorem 5.1(c)]{memoli2011-gw}, noting the difference in normalization of $\gw_p$ in the current paper).

\paragraph{Relaxing the Gromov-Wasserstein distance.} The registration properties of GW distances described above (i.e., through probabilistic matchings in the guise of optimal couplings) are vulnerable to outliers in the data. As an example, suppose that in the process of acquiring pmm-space datasets $X$ and $Y$, measurement error leads to additional `noisy' points being included in the supports of $\mu_X$ and $\mu_Y$; since the formulation of GW distance involves couplings which match the measures exactly (i.e., exact couplings), the resulting registration will necessarily fit the noise. This observation motivates one to alter the formulation of the distance to allow for relaxed notions of coupling that avoid the noise-fitting issue---indeed, several notions of \textit{partial} or \textit{unbalanced} Gromov-Wasserstein distance have been introduced in the recent literature to address this issue, and related issues of non-robustness \cite{chapel2020partial,sejourne2021-conic,Vanderbilt2024-PGW,kong2024outlier}.

We now introduce a variant of GW which is relaxed in a rather straightforward manner: we simply expand the feasible set to allow $\epsilon$-relaxed couplings, as in Definition \ref{def:relaxed_couplings}. 

\begin{definition}[Partial Gromov-Wasserstein Distances]\label{def:partial_GW}
    Let $p \in [1,\infty]$, let $\epsilon = (\epsilon_1,\epsilon_2) \in [0,\infty]^2$, and let $X$ and $Y$ be pmm-spaces.  We define the associated \define{partial Gromov-Wasserstein $p$-distance} between $X$ and $Y$ to be 
    \[
    \pgw_{\eps,p}(X,Y) \coloneqq \inf_{\pi \in \C_{\eps}(\mu_X,\mu_Y)}\norm{d_X - d_Y}_{L^p(\pi\otimes \pi)}.
    \]
    Similarly, we define the \define{symmetrized partial Gromov-Wasserstein $p$-distance} to be 
    \[
    \spgw_{\eps,p}(X,Y) \coloneqq \inf_{\pi \in \Cs_{\eps}(\mu_X,\mu_Y)}\norm{d_X - d_Y}_{L^p(\pi\otimes \pi)}.
    \]
\end{definition}

\begin{remark}
    The construction of $\pgw$ is a generalization of the \textit{partial Gromov-Wasserstein problem} proposed by Chapel et al.\ in \cite{chapel2020partial}, as we explain in more detail below in Remark \ref{rem:special cases} and Section \ref{sect: relation to chapel}. We note that the Chapel et al.\ version is defined on the space of general mm-spaces, whose measures have arbitrary total mass. The justification for our reformulation is that it makes certain proofs easier, and allows more immediate comparisons to other notions of partial GW distance in the literature---see Section \ref{sect: related works}. 
\end{remark}

Recall the comment after Definition \ref{def:relaxed_couplings}. The intuition is that $\pgw_{\eps,p}$ and $\spgw_{\eps,p}$ behave similarly to $\gw_p$, with the difference being that the couplings allow for controlled rearrangement of mass. If the relaxation level is appropriately matched to the noise level in the data, this setup is intended to allow the distances to mollify the effect of outliers in the inferred registrations. 

\begin{remark}
    We note that we refer to $\pgw_{\eps,p}$ and $\spgw_{\eps,p}$ as `distances' only in a colloquial sense: they are not, in general, pseudometrics, as we establish below in Section \ref{sec:metric_properties_spgw}. To be more precise, we could refer to them as \textit{divergences} or \textit{dissimilarity functions}, but we prefer to stick with the \textit{distance} terminology, which we find to be more linguistically natural.
\end{remark}

\begin{remark}\label{rem:special cases}
    The following describes some special cases for the relaxed distances $\pgw$ and $\spgw$.
    \begin{enumerate}[(A)]
        \item For $\eps = (0,0)$, 
        \[
        \pgw_{(0,0),p}(X,Y) = \spgw_{(0,0),p} = \gw_p(X,Y).
        \]
        Indeed, this follows immediately from Remark \ref{rmk:relaxed_at_zero}.
       
        \item 
        For $\eps= (\eps_1,\eps_2) = (0,\infty)$,
        \begin{equation*}
        \eqindent
        \spgw_{(0,\infty),p}(X,Y) = \inf\set{
        \norm{d_X - d_Y}_{L^p(\pi\otimes \pi)} : \pi \in \P(X\times Y),\, \pi_X = \mu_X
        }.
    \end{equation*}
    That is, the marginal $\pi_X$ is strictly required to be $\mu_X$, while the marginal $\pi_Y$ is unrestricted. This recovers the \textit{semi-relaxed Gromov-Wasserstein distance}, first introduced in \cite{vincent2021semi}. In \cite{clark2024generalized}, the authors also showed that $\spgw_{(0,\infty),p}$ satisfies the triangle inequality if and only if $p = \infty$, and that, in this case, a symmetrized version agrees with the \textit{modified Gromov-Hausdorff distance} studied in \cite{memoli2012some}. See also Theorem \ref{thm:relaxed triangle} below.
    \item When $\eps_1=\eps_2 \in (0,\infty)$ and $p=2$, the distance is essentially the same as the partial Gromov-Wasserstein distance introduced by Chapel et al.\ in \cite{chapel2020partial}. The exact connection between the construction of Chapel et al.\ and those of the present paper is elucidated below in Section~\ref{sect: relation to chapel}.
    \end{enumerate}
\end{remark}

\subsection{Basic properties of \texorpdfstring{$\pgw$}{PGW} and \texorpdfstring{$\spgw$}{SPGW}}

It is natural to ask whether the properties of GW distance described in Theorem \ref{thm:GW_properties} extend to the partial GW variants introduced in Definition \ref{def:partial_GW}. We show in this subsection that the existence of optimal (relaxed) couplings (i.e., the analogue of Theorem \ref{thm:GW_properties}(A)) does hold. The metric properties of $\pgw$ and $\spgw$ (cf.\ Theorem \ref{thm:GW_properties}(B)) are more subtle, and a precise discussion is deferred to Section \ref{sec:metric_properties_spgw}. We also show in this subsection that the GW distance is recovered as a limit of partial GW distances as the relaxation parameters decay to zero.

\paragraph{Existence of optimal couplings.} We now show that the infima in the optimization problems defining $\pgw_{\eps,p}$ and $\spgw_{\eps,p}$ are always realized by optimal couplings; that is, infima are actually minima. This is the analog of Theorem \ref{thm:GW_properties}(A) in the GW setting. The proof is more subtle than in the case of exact couplings and requires an additional ingredient for the case $p = \infty$. Before we state the result, we introduce the notion of distortion to simplify notation.

\begin{definition}[\hspace{-1sp}\cite{chowdhury2019gromov}]\label{def:distortion}
    
For $p \in [1,\infty]$ and pmm-spaces $X$ and $Y$, we define the $p$-distortion of a coupling $\pi \in \P(X\times Y)$ to be
\begin{equation*}
    \dis_p(\pi):= \norm{d_X - d_Y}_{L^p(\pi\otimes \pi)}.
\end{equation*}
\end{definition}

\begin{lemma}\label{lem:distortion-continuous}
    Fix $\eps \geq 0$ and pmm-spaces $X$ and $Y$. 
    \begin{enumerate}[(A)]
        \item For $1\leq p < \infty$, the function $\dis_p(\cdot)$ is continuous on $\C_\eps(\mu_X,\mu_Y)$ and on $\Cs_\eps(\mu_X,\mu_Y)$. 
        \item For $p = \infty$, $\dis_p(\cdot)$ is lower-semicontinuous on $\C_\eps(\mu_X,\mu_Y)$ and on $\Cs_\eps(\mu_X,\mu_Y)$.
    \end{enumerate}
\end{lemma}

The proof of Lemma \ref{lem:distortion-continuous} is almost identical to the proof of \cite[Lemma 11]{chowdhury2019gromov}. We omit the details here.

\begin{theorem}[Existence of Optimal Relaxed Couplings]\label{thm:realization}
    For every $\eps = (\eps_1,\eps_2)$, $1\leq p \leq \infty$, and compact pmm-spaces $X,Y$, there exist $\pi \in \C_{\eps}(\mu_X,\mu_Y)$ and $\gamma \in \Cs_{\eps}(\mu_X,\mu_Y)$ such that 
    \[
    \norm{d_X - d_Y}_{L^p(\pi\otimes \pi)} = \pgw_{\eps,p}(X,Y) \quad \mbox{and} \quad \norm{d_X -d_Y}_{L^p(\gamma\otimes \gamma)} = \spgw_{\eps,p}(X,Y).
    \]
\end{theorem}

\begin{proof}
    Let $(\pi_k) \subset \C_{\eps}(\mu_X,\mu_Y)$ and $(\gamma_k)\subset \Cs_\eps(X,Y)$ be  sequences of relaxed couplings such that 
    \begin{equation*}
        \norm{d_X - d_Y}_{L^p(\pi_k\otimes \pi_k)} \to \pgw_{\eps,p}(X,Y)
        \text{ and }
        \norm{d_X - d_Y}_{L^p(\gamma_k\otimes \gamma_k)} \to \spgw_{\eps,p}(X,Y).
        \label{eq:realization-conv}
    \end{equation*}
    By Prokhorov's Theorem, we can extract a subsequence from each, still called $(\pi_k)$ and $(\gamma_k)$, with 
    \begin{equation*}
        \pi_k \to \pi \in \P(X\times Y)
        \text{ and }
        \gamma_k \to \gamma \in \P(X\times Y) \text{ weakly}.
    \end{equation*}

    First, we verify that 
    \begin{equation}
    \pi \in \C_{\eps}(\mu_X,\mu_Y)
    \text{ and }
    \gamma \in \Cs_\eps(\mu_X,\mu_Y).
        \label{eq:realization-0}
    \end{equation}
    We will only check that $\pi_{X} := (p^X)_\sharp \pi \leq (1+\eps_1)\mu_X$, since the rest are similar (the inner regularity argument for $\C_{\eps}$ below is replaced with both inner and outer regularity for $\Cs_\eps$ coupling). 
    
    By definition, $\pi_k \in \C_\eps(\mu_X,\mu_Y)$ implies that $(p^X)_\sharp \pi_k \leq (1+\eps_1)\mu_X$. Let $A \subset X$ be Borel and let $\eta > 0$. By the outer regularity of $\mu_X$, there exists a slightly enlarged neighborhood $A_\eta \supset A$  satisfying
    \begin{equation}
        \mu_X(A_\eta \setminus A) \leq \eta.
        \label{eq:realization-eta}
    \end{equation}
    Thanks to Urysohn's Lemma, there exists a continuous function $\phi$ on $X\times Y$ satisfying
    \begin{equation}
        \mathbb{I}_{A\times Y} \leq \phi \leq \mathbb{I}_{A_\eta\times Y}
        \label{eq:realization-indicator}
    \end{equation}
    where $\mathbb{I}_S$ is the indicator function of a set $S$.
    On the one hand, \eqref{eq:realization-indicator} implies
    \begin{equation}
        \pi_X(A) = \pi(A\times Y) = \int_{X\times Y} \mathbb{I}_{A\times Y} d\pi \leq \int_{X\times Y} \phi d\pi.
        \label{eq:realization-1}
    \end{equation}
    On the other hand, we can use \eqref{eq:realization-eta} and the marginal control on the $\pi_k$'s to obtain
    \begin{equation}
       \begin{split}
            \int_{X\times Y} \phi d\pi 
            &= \lim_{k\to\infty}\int_{X\times Y} \phi d\pi_k 
            \\
            &= \lim_{k \to \infty}\brac{
            \int_{X\times Y} (\phi - \mathbb{I}_{A\times Y})d\pi_k
            + \pi_k(A\times Y)
            }\\
            &\leq (1+\eps_1)\mu_X(A) + \lim_{k\to\infty}\int_{(A_\eta\setminus A)\times Y} d\pi_k\\
            &\leq (1+\eps_1)\mu_X(A) + (1+\eps_1)\mu_X(A_\eta \setminus A)\\
            &\leq (1+\eps_1)\mu_X(A) + (1+\eps_1)\eta.
       \end{split}
       \label{eq:realization-2}
    \end{equation}
    Since $A\subset X$ and $\eta > 0$ are arbitrary, we can conclude from \eqref{eq:realization-1} and \eqref{eq:realization-2} that $\pi_X \leq (1+\eps_1)\mu_X$.

    Next, we show that $\pi$ and $\gamma$ are optimal. If $1 \leq p < \infty$, weak convergence implies that
    \begin{equation}
        \norm{d_X - d_Y}_{L^p(\pi\otimes \pi)} = \pgw_{\eps,p}(X,Y)
        \text{ and }
        \norm{d_X - d_Y}_{L^p(\gamma\otimes \gamma)} = \spgw_{\eps,p}(X,Y).
        \label{eq:realization-prokh-1}
    \end{equation}
    Suppose $p = \infty$. Thanks to Lemma \ref{lem:distortion-continuous}(B), we have
    \begin{equation*}
        \begin{split}
            \pgw_{\eps,\infty}(X,Y) &= \liminf_{k \to \infty}\dis_\infty(\pi_k) \geq \dis_\infty(\pi) = \norm{d_X - d_Y}_{L^\infty(\pi\otimes \pi)}
        \text{, and }\\
        \spgw_{\eps,\infty}(X,Y) &= \liminf_{k \to \infty}\dis_\infty(\gamma_k) \geq \dis_\infty(\gamma) = \norm{d_X - d_Y}_{L^\infty(\gamma\otimes \gamma)}.
        \end{split}
    \end{equation*}
    By the minimality of $\pgw_{\eps,\infty}$ and $\spgw_{\eps,\infty}$, we can conclude that
    \begin{equation}
        \norm{d_X - d_Y}_{L^\infty(\pi\otimes \pi)} = \pgw_{\eps,\infty}(X,Y)
        \text{ and }
        \norm{d_X - d_Y}_{L^\infty(\gamma\otimes \gamma)} = \spgw_{\eps,\infty}(X,Y).
        \label{eq:realization-prokh-2}
    \end{equation}

    The conclusion of the proposition follows from \eqref{eq:realization-0}, \eqref{eq:realization-prokh-1}, and \eqref{eq:realization-prokh-2}.
\end{proof}

\paragraph{Convergence in parameters.} 
We move to investigate the dependence of partial GW distances on the relaxation parameters. 
We begin with a fundamental result which states that the maps $\eps \mapsto \pgw_{\eps,p}(\cdot,\cdot)$ and $\eps \mapsto \spgw_{\eps,p}(\cdot,\cdot)$ are lower semicontinuous in $\eps$ (Theorem \ref{thm:lower-semicontinuity} below). By Remark \ref{rem:special cases} (A), the partial GW distances agree with the GW distance when the relaxation parameters vanish, i.e., when $\eps = (0,0)$. Theorem \ref{thm:GW-convergence} below demonstrates the intuitive result that the distances converge as the relaxation parameters decay to zero. Finally, Theorem \ref{thm:monotone-convergence} below states that the partial GW distances always converge with respect to monotone-decreasing parameters (see the definition of monotone-decreasing before Theorem \ref{thm:monotone-convergence}).

\begin{theorem}[Lower Semicontinuity in the Relaxation Parameters]\label{thm:lower-semicontinuity}
    Let $p \in [1,\infty]$ and let $X,Y$ be compact pmm-spaces. Then the maps $[0,\infty)^2 \mapsto \pgw_{\eps,p}(X,Y)$ and $[0,\infty)^2 \mapsto \spgw_{\eps,p}(X,Y)$ are both lower semicontinuous. 
\end{theorem}

\begin{proof}
    We prove the theorem for $\pgw_{\eps,p}$. The case for $\spgw_{\eps,p}$ is similar, 
    
    Fix an arbitrary sequence $\eps_n \to \eps_0$. We need to show that
    \begin{equation}
\liminf_{n\to\infty}\pgw_{\eps_n,p}(X,Y)\geq \pgw_{\eps_0,p}(X,Y).
        \label{eq:eps-continuity-2}
    \end{equation}

    Set $\delta_n := \pgw_{\eps_n,p}(X,Y)$ and consider an arbitrary convergent subsequence $(\delta_{n_k})\subset (\delta_n)$. Let $\pi_k \in \C_{\eps_{n_k}}(\mu_X,\mu_Y)$ be a measure realizing $\delta_{n_k} = \pgw_{\eps_{n_k},p}(X,Y)$ for each $k$. By Prokhorov's theorem, we may replace $\pi_k$ by a subsequence, denoted $\pi_{k_{j}}$, such that $\pi_{k_j} \to \pi \in \P(X\times Y)$ weakly. We also have
    \begin{equation}
        \pi \in \C_{\eps_0}(X,Y) \text{ and }
        \pgw_{\eps_0,p}(X,Y)\leq \norm{d_X - d_Y}_{L^p(\pi\otimes \pi)}.
        \label{eq:eps-continuity-5}
    \end{equation}

    For $1 \leq p < \infty$, the weak convergence of $(\pi_{k_j})$ implies
    \begin{equation}
        \lim_{j\to\infty}\delta_{n_{k_j}} = \lim_{j\to\infty}\pgw_{\eps_{n_{k_j}},p}(X,Y) = 
        \lim_{j\to\infty}\norm{d_X - d_Y}_{L^p(\pi_{k_j}\otimes\pi_{k_j})}
        = \norm{d_X - d_Y}_{L^p(\pi\otimes \pi)}.
        \label{eq:eps-continuity-3}
    \end{equation}
    For $p = \infty$, Lemma \ref{lem:distortion-continuous}(B) implies
    \begin{equation}
        \lim_{j\to\infty}\delta_{n_{k_j}} = \lim_{j \to \infty}\pgw_{\eps_{n_{k_j}},\infty}(X,Y) = \liminf_{j\to\infty}\dis_\infty(\pi_{k_j}) \geq \dis_\infty(\pi) = \norm{d_X - d_Y}_{L^\infty(\pi\otimes \pi)}.
        \label{eq:eps-continuity-3'}
    \end{equation}
    
    Since $(\delta_{n_k})_{k=1}^{\infty}$ is convergent, we conclude from \eqref{eq:eps-continuity-3} and \eqref{eq:eps-continuity-3'} that, for $p \in [1,\infty]$,
    \begin{equation}
        \lim_{k \to \infty}\delta_{n_k} = \lim_{k \to \infty}\pgw_{\eps_{n_k}, p}(X,Y) \geq \norm{d_X - d_Y}_{L^p(\pi \otimes \pi)}.
        \label{eq:eps-continuity-4}
    \end{equation}
We see from \eqref{eq:eps-continuity-5} and
\eqref{eq:eps-continuity-4} that
    \begin{equation}
        \lim_{k\to\infty}\pgw_{\eps_{n_k},p}(X,Y) \geq \pgw_{\eps_0,p}(X,Y).
    \end{equation}
    Since the choice of the convergent subsequence $(\delta_{n_k})_{k = 1}^{\infty}$ is arbitrary, we see that \eqref{eq:eps-continuity-2} holds. 
\end{proof}

\begin{theorem}[Convergence to $\gw_p$]\label{thm:GW-convergence}
    Let $p \in [1,\infty]$ and let $X,Y$ be compact pmm-spaces. The maps $[0,\infty)^2 \mapsto \pgw_{\eps,p}(X,Y)$ and $[0,\infty)^2 \mapsto \spgw_{\eps,p}(X,Y)$ are continuous at $\eps = (0,0)$. Moreover,
    \begin{equation}
    \lim_{n\to\infty}\pgw_{\eps_n,p}(X,Y)= \gw_p(X,Y) = \lim_{n\to\infty}\spgw_{\eps_n,p}(X,Y).
    \end{equation}
\end{theorem}

\begin{proof}
    We prove the theorem for $\pgw_{\eps,p}$. The case for $\spgw_{\eps,p}$ is similar. Let $\eps_n \to (0,0)$. Since $\pgw_{\eps,p}(X,Y) \geq \pgw_{(0,0)}(X,Y)$, we have
    \begin{equation}
    \limsup_{n\to\infty}\pgw_{\eps_n,p}(X,Y)\leq \pgw_{(0,0)}(X,Y) = \gw_p(X,Y).
    \label{eq:limsup-lowerbound-1}
    \end{equation}
where the last equality is due to Remark \ref{rem:special cases}(A). The theorem follows from \eqref{eq:eps-continuity-2} in the proof of Theorem \ref{thm:lower-semicontinuity} and \eqref{eq:limsup-lowerbound-1}.
\end{proof}


Let $\eps_n = (\eps_n^{(1)}, \eps_n^{(2)})$ be a sequence in $\R^2$. We say $\eps_n$ is monotone-decreasing if $\eps_n^{(i)} \geq \eps_{n+1}^{(i)}$ for all $n$ and $i =1,2$, i.e., $\eps_n$ is monotone decreasing in each component.

\begin{theorem}[Monotone Convergence in the Relaxation Parameters]\label{thm:monotone-convergence}
    Let $p \in [1,\infty]$.
    Let $\eps_n \in [0,\infty)^2$ be a monotone-decreasing sequence converging to $\eps_0 \in [0,\infty)^2$ and let $X,Y$ be compact pmm-spaces. Then
    \begin{align}
        \lim_{n\to\infty}\pgw_{\eps_n,p}(X,Y)&= \pgw_{\eps_0,p}(X,Y),\text{ and }
        \label{eq:eps-continuity-0}
        \\
        \lim_{n\to\infty}\spgw_{\eps_n,p}(X,Y) &= \spgw_{\eps_0,p}(X,Y).
        \label{eq:eps-continuity-0sym}
    \end{align}
\end{theorem}

\begin{proof}
    We will prove \eqref{eq:eps-continuity-0} only. The proof of \eqref{eq:eps-continuity-0sym} is similar. Since $\eps_n$ is monotone-decreasing, we see from the definition of $\C_{\eps}$ that $\C_{\eps_n}(\mu_X,\mu_Y)\supset \C_{\eps_{n+1}}(\mu_X,\mu_Y)$, so $\pgw_{\eps_n,p}(X,Y)\leq \pgw_{\eps_0,p}(X,Y)$ for all $n$, so 
\begin{equation}
\limsup_{n\to\infty}\pgw_{\eps_n,p}(X,Y)\leq \pgw_{\eps_0,p}(X,Y).
\label{eq:eps-continuity-1}
\end{equation}
The theorem follows from \eqref{eq:eps-continuity-2} in the proof of Theorem \ref{thm:lower-semicontinuity} and \eqref{eq:eps-continuity-1}.
    
\end{proof}

\section{Approximate metric properties for \texorpdfstring{$\pgw$}{PGW} and \texorpdfstring{$\spgw$}{sPGW}}\label{sec:metric_properties_spgw}

When the partial Gromov-Wasserstein problem was introduced in \cite{chapel2020partial} (as referenced in the introduction; see Definition \ref{def:Chapel} for details), its metric properties, such as whether it satisfies the triangle inequality, were not addressed (to be clear, \cite{chapel2020partial} makes no claims that their framework induces a pseudometric, and only considers it as an optimization problem with interesting applications). An initial motivation for the present article was to study the metric structure of the partial Gromov-Wasserstein problem, which is equivalent to studying it for $\pgw_{\eps,p}$---see  Proposition \ref{prop:equivalence} below.

In this section, we study the metric properties of $\pgw$ and $\spgw$. As a trivial point, it is clear that if $\eps_1 = \eps_2$, then $\pgw_{\eps,p}(X,Y) = \pgw_{\eps,p}(Y,X)$ for all $X$ and $Y$. This section will therefore focus on non-degeneracy and triangle inequalities for the distances. We show that both properties fail, but give more refined results which demonstrate that they do satisfy approximate versions.

\subsection{Approximate Non-Degeneracy}

The next theorem gives a characterization of the partial Gromov-Wasserstein distances in terms of Gromov-Wasserstein distances between pmm-spaces with adjusted mass. It will be followed by a corollary that addresses the extent to which the distance $\pgw$ is \textit{degenerate}, in the sense that non-isomorphic pmm-spaces can have zero distance. To state the theorem, we introduce some additional concepts.

\begin{definition}[Eccentricity and Circumradius; c.f.\ {\cite[Definition 5.3]{memoli2011-gw}} and {\cite[Definition 2.1]{memoli2012some}}]\label{def:ecc and rad}
    Let $X = (X,d_X,\mu_X)$ be an mm-space and $p \in [1,\infty]$. The \define{$p$-eccentricity function of $X$} is the function
    \begin{align*}
        \mathsf{ecc}_p^X:X &\to \R \\
        x &\mapsto \|d_X(x,\cdot)\|_{L^p(\mu_X)}.
    \end{align*}
    We define the \define{$p$-circumradius of $X$} to be the quantity
    \[
    \mathsf{rad}_p(X) = \inf_{x \in X} \mathsf{ecc}_p^X(x).
    \]
\end{definition}

\begin{definition}[$\eps$-Relations]\label{def:eps-relation}
    Let $X = (X,d_X,\mu_X)$ be a pmm-space and let $\Tilde{X} = (X,d_X,\Tilde{\mu}_X)$ be another mm-space with the same underlying metric space. For $\eps \in [0,\infty]$, we write $\Tilde{X} \sim_\eps X$ if $\Tilde{\mu}_X \leq (1+\eps) \mu_X$. Similarly, we write $\Tilde{X} \sim_\eps^s X$ if
    \[
    \left\| \frac{d\tilde{\mu}_X}{d\mu_X}\right\|_{L^\infty(\mu_X)}, \left\| \frac{d\mu_X}{d\Tilde{\mu}_X}\right\|_{L^\infty(\Tilde{\mu}_X)} \leq 1 + \eps.
    \]
\end{definition}

\begin{remark}
    The relations $\sim_\eps$ and $\sim_\eps^s$ are not equivalence relations. Both relations are reflexive, only $\sim_\eps^s$ is symmetric, and neither is transitive.
\end{remark}

\begin{theorem}\label{thm:approxi-nondegen}
    Let $X$ and $Y$ be compact pmm-spaces, let $\eps = (\eps_1,\eps_2) \in [0,\infty]^2$, let $p \in [1,\infty]$ and let $q$ satisfy $1/q+1/p = 1$ (setting $1/\infty = 0$).
    Suppose that $\pgw_{\eps,p}(X,Y)$ is realized by $\pi \in \mathcal{C}_\eps(\mu_X,\mu_Y)$ (see Theorem \ref{thm:realization}). Define pmm-spaces $X_\pi = (X,d_X,\pi_X)$ and $Y_\pi = (Y,d_Y,\pi_Y)$. Then
    \begin{equation}
        \pgw_{\eps,p}(X,Y) = \gw_p(X_\pi,Y_\pi) = \inf_{\substack{\Tilde{X} \sim_{\eps_1} X \\ \Tilde{Y} \sim_{\eps_2} Y}} \gw_p(\Tilde{X},\Tilde{Y}).
        \label{eq:thm:nondegen-2}
    \end{equation}
    Similarly, if $\spgw_{\eps,p}(X,Y)$ is realized by $\pi \in \mathcal{C}_\eps(\mu_X,\mu_Y)$, then 
    \[
    \spgw_{\eps,p}(X,Y) = \gw_p(X_\pi,Y_\pi) = \inf_{\substack{\Tilde{X} \sim_{\eps_1}^s X \\ \Tilde{Y} \sim_{\eps_2}^s Y}} \gw_p(\Tilde{X},\Tilde{Y}).
    \]
    Moreover, in either case (with $\mathsf{rad}_p(\cdot)$ as in Definition \ref{def:ecc and rad}),
    \begin{equation}
    \gw_p(X_\pi,X) \leq 2 \cdot 2^{1/q}\eps_1^{1/p}\mathsf{rad}_p(X),\quad \gw_p(Y_\pi,Y) \leq 2 \cdot 2^{1/q}\eps_2^{1/p}\mathsf{rad}_p(Y).
    \label{eq:thm:nondegen-1}
    \end{equation}
\end{theorem}

\begin{remark}\label{rem:p_infinity_case}
    The estimates on $\gw_p(X_\pi,X)$ and $\gw_p(Y_\pi,Y)$ when $p=\infty$ become
    \[
    \gw_\infty(X_\pi,X) \leq 4 \cdot \mathsf{rad}_\infty(X), \quad \gw_\infty(Y_\pi,Y) \leq 4 \cdot \mathsf{rad}_\infty(Y).
    \]
    The loss of dependence on $\eps_1$ and $\eps_2$ causes these estimates to be rather poor. We include the $p=\infty$ case here primarily because the above observation will be useful later in Theorem \ref{thm:equivalence_of_topologies}. 
\end{remark}

We need the following as an intermediate estimate.

\begin{theorem}[\hspace{1sp}{\cite[Theorem 6.15]{villani-ot-book}}]\label{thm:diam-bound}
    Let $1\leq p < \infty$ and $1/p + 1/q = 1$. Then, for any $x \in X$,
    \begin{equation}
        \w_p(\mu,\nu) \leq 2^{1/q} \brac{\int_X
        d(x,y)^pd\abs{\mu-\nu}(y)
        }^{1/p},
        \label{eq:diam-bound}
    \end{equation}
    where $|\mu-\nu|$ denotes the total variation measure.
\end{theorem}

\begin{proof}[Proof of Theorem \ref{thm:approxi-nondegen}]
    We will prove the claims for $\pgw_{\eps,p}$, since the proofs for $\spgw_{\eps,p}$ are almost identical. Recall the distortion functional $\dis_p(\cdot)$ from Definition \ref{def:distortion}. To derive the identity \eqref{eq:thm:nondegen-2}, we first observe that $\pi$ must be an optimal coupling for $\gw_p(X_\pi,Y_\pi)$. Indeed, if $\pi' \in \mathcal{C}(\pi_X,\pi_Y)$ satisfied $\mathsf{dis}_p(\pi') < \mathsf{dis}_p(\pi)$, then this would contradict the optimality of $\pi$ for $\pgw_{\eps,p}$, since $\pi' \in \mathcal{C}_\eps(\mu_X,\mu_Y)$. This implies
    \[
    \gw_p(X_\pi,Y_\pi) = \mathsf{dis}_p(\pi) = \pgw_p(X,Y).
    \]
    For the second equality, we clearly have $\gw_p(X_\pi,Y_\pi) \geq \inf_{\Tilde{X}, \Tilde{Y}} \gw_p(\Tilde{X},\Tilde{Y})$. On the other hand, let $\Tilde{X} \sim_{\eps_1} X$ and $\Tilde{Y} \sim_{\eps_2} Y$ be arbitrary. For any optimal coupling $\Tilde{\pi} \in \mathcal{C}(\Tilde{\mu}_X,\Tilde{\mu}_Y)$ for $\gw_p(\Tilde{X},\Tilde{Y})$, we have $\Tilde{\pi} \in \mathcal{C}_\eps(X,Y)$, so that
    \[
    \gw_p(\Tilde{X},\Tilde{Y}) = \mathsf{dis}_p(\Tilde{\pi}) \geq \mathsf{dis}_p(\pi) = \pgw_p(X,Y) = \gw_p(X_\pi,Y_\pi),
    \]
    by the optimality of $\pi$ for $\pgw_{\eps,p}(X,Y)$ and by the work above. Since $\Tilde{X}$ and $\Tilde{Y}$ were arbitrary, this shows $\gw_p(X_\pi,Y_\pi) = \inf_{\Tilde{X}, \Tilde{Y}} \gw_p(\Tilde{X},\Tilde{Y})$.
    
    It remains to establish the estimates \eqref{eq:thm:nondegen-1}. We do so for $p<\infty$ for convenience; the $p=\infty$ cases follow by a limiting argument. The bound \eqref{eq:GWp-Wp}, together with Theorem \ref{thm:diam-bound}, implies that
    \begin{equation}
        \gw_p(X_\pi, X)\leq 2\cdot \w_p(\pi_X,\mu_X)\leq 2 \cdot 2^{1/q}\left(\int d_X(x,y)^p d|\pi_X-\mu_X|(y)\right)^{1/p}
        \label{eq:nondegen-proof1}
    \end{equation}
    holds for any $x \in X$. Let $\rho_X := \frac{d\pi_X}{d\mu_X}$. Thanks to the definition of $\C_\eps$, we have
    \begin{equation}
        \norm{\rho_X - 1}_{L^\infty(\mu_X)} \leq \eps_1.
         \label{eq:nondegen-proof2}
    \end{equation}
    Using \eqref{eq:nondegen-proof2}, we have, for any $x\in X$,
    \begin{equation}
        \int_X d_X(x,y)^p d\abs{\pi_X- \mu_X}(y) \leq \int_X d_X(x,y)^p \abs{\rho_X - 1}d\mu_X(x) \leq \eps_1 \mathsf{ecc}_p^X(x)^p.
        \label{eq:nondegen-proof3}
    \end{equation}
    The first estimate in \eqref{eq:thm:nondegen-1} follows from \eqref{eq:nondegen-proof1},  \eqref{eq:nondegen-proof3}, and the definition of $\mathsf{rad}_p(X)$. The second estimate in \eqref{eq:thm:nondegen-1} is proved similarly.  
\end{proof}

As an easy corollary, we get a characterization of the non-degeneracy of the partial Gromov-Wasserstein distances. It roughly states that if $\pgw_{\eps,p}(X,Y) = 0$ or
$\spgw_{\eps,p}(X,Y) = 0$, then $X$ and $Y$ have large portions (depending on the size of the parameters $\eps = (\eps_1,\eps_2)$) that are isomorphic. 

\begin{corollary}
\label{cor: PGW equal 0}
    Using the notation of Theorem \ref{thm:approxi-nondegen}, $\pgw_{\eps,p}(X,Y) = 0$ if and only if $X_\pi$ and $Y_\pi$ are isomorphic. Similarly, $\spgw_{\eps,p}(X,Y) = 0$ if and only if $X_\pi$ and $Y_\pi$ are isomorphic (where $\pi$ is now optimal for $\spgw$). Moreover, the spaces $X_\pi$ and $Y_\pi$ satisfy the estimates \eqref{eq:thm:nondegen-1}. 
\end{corollary}

\subsection{Approximate Triangle Inequality for \texorpdfstring{$\spgw$}{sPGW}}

 We now show that the `distance' $\spgw_{\eps,p}$ does not, in fact, enjoy a triangle inequality. The definition of the symmetrized partial GW distance $\spgw_{\epsilon,p}$ was partially motivated as a means to obtain a weakened version of the triangle inequality, which we establish below in the main result of this subsection, Theorem \ref{thm:relaxed triangle}. Before stating the theorem, we provide a counterexample showing that a true triangle inequality does not hold. A similar construction works for $\pgw$ as well, but we only focus on $\spgw$, since it is the more rigid of the two.

\paragraph{Counterexample.}
   For finite metric spaces, it is convenient to use matrix notations. The probability measure is given by a stochastic vector and the distance is given by a zero-diagonal entry-wise nonnegative symmetric matrix. Let $\delta\in (0,\frac{1}{6}]$ and $\beta= \frac{1}{2}-\delta$. Let $\eps= (0,\eps_2)$ for some $\eps_2 \in [0,1)$ satisfying
   \begin{equation*}
       \frac{\delta}{\frac{1}{2}-\delta} \leq \eps_2 \leq \frac{\frac{1}{2}-\delta}{\frac{1}{2}+\delta}.
   \end{equation*}
   Consider three finite mm-spaces $(X, \mu_X, d_X)$, $(Y, \mu_Y, d_Y)$, and $(Z, \mu_Z, d_Z)$ defined by
   \begin{itemize}
       \item $X = \{x\} $ and $ \mu_X = (1)$;
         \item $Y = \{y_1,y_2\}$, $ \mu_Y= (\beta, 1-\beta)$, and $D_Y = \begin{bmatrix}
    0 & 1\\ 1 & 0
\end{bmatrix}$;
         \item $Z = \{z_1,z_2\}$, $ \mu_Z = (\frac{1}{2},\frac{1}{2})$, and $D_Z = \begin{bmatrix}
    0 & 1\\ 1 & 0
\end{bmatrix}$.
   \end{itemize}
   Set $\pi:= \mathsf{Diag}(1/2,1/2) \in \P(Z\times Y)$.
   It is clear that $\pi_Z = \mu_Z$, and we can verify by hand that $\frac{1}{1+\eps_2} \leq \frac{\pi_Y}{\mu_Y}\leq 1+\eps_2$. Therefore,
 \begin{equation*}
     \spgw_{\eps,p}(Z,Y)^p \leq \sum_{i,i', j, j' = 1}^{2} \abs{D_Z(i,i') - D_Y(j,j')}^p \pi(i,j)\pi(i',j') = 0.
 \end{equation*}
Similarly one  can  define $\Bar{\pi}:= \mathsf{Diag}(\beta,1-\beta)\in \P(Y\times Z)$ and verify that  $\Bar{\pi}_Y=\mu_Y$  and $\frac{1}{1+\eps_2} \leq \frac{\Bar{\pi}_Y}{\mu_Z}\leq 1+\eps_2$. Hence $\spgw_{\eps,p}(Y,Z) = 0$.

Every $\gamma\in \Cs_{\eps}(\mu_X,\mu_Y)$ can be written as $\gamma = \begin{bmatrix}
    s & (1-s)
\end{bmatrix}$
with $\frac{1}{1+\eps_2}\leq \frac{s}{\beta},\frac{1-s}{1-\beta} \leq 1+\eps_2$. 
An elementary but tedious computation yields
\begin{equation}
    \spgw_{\eps,p}(X,Y)^p = \frac{1 -2\delta}{1+\eps_2}\brac{1-\frac{1 -2\delta}{2(1+\eps_2)}}, 
\end{equation}
where the distance is achieved by $\gamma = \begin{bmatrix}
    s & (1-s)
\end{bmatrix}$ with $s = \frac{\frac12 -\delta}{1+\eps_2}$.

Similarly, one can verify that every $\gamma \in \Cs_{\eps}(\mu_X,\mu_Z)$ has the form $\gamma = \begin{bmatrix}
     (\frac{1}{2}(1-t) & \frac{1}{2}(1+t)
 \end{bmatrix}
 $ with $\abs{t} \leq \frac{\eps_2}{1+\eps_2}$. A direct computation yields 
 \begin{equation*}
     \spgw_{\eps,p}(X,Z)^p  = \frac{1}{1+\eps_2}\brac{1-\frac{1}{2(1+\eps_2)}}.
 \end{equation*}
In view of the computations above, we see that $\spgw_{\eps,p}(X,Z) >  \spgw_{\eps,p}(X,Y) + \spgw(Y,Z)$.

\paragraph{Relaxed Triangle Inequality.} We now show that $\spgw$ satisfies a certain relaxed triangle inequality. We begin by recalling a standard result from measure theory \cite[Chapter 3]{kallenberg1997probabilitybook}.

\begin{theorem}[Disintegration of Measures]\label{thm:distintegration}
    Let $\Pi,Y$ be compact probability metric measures spaces, let $f: \Pi \to Y$ be measurable, and let $\pi \in \P(\Pi)$. Let $\mu = f_\sharp\pi \in \P(Y) $. Then there exists a family $\set{\pi(\cdot|y)}_{y \in Z}$ of Borel probability measures such that $\pi(\cdot|y) \in \P(f^{-1}(y))$ and
    \begin{equation}
        \int_{\Pi}g(\xi)d\pi(\xi) = \int_Y \int_{f^{-1}(y)}g(\xi)d\pi(\xi|y)d\mu(y)
    \end{equation}
    for all Borel-measurable maps $g: \Pi\to [0,\infty]$. Moreover, $\set{\pi(\cdot|y)}_{y \in Y}$ is $\mu$-almost everywhere unique. 
\end{theorem}

A key result in classical optimal transport theory is the \textit{gluing lemma}, which gives a natural notion of composing (exact) couplings and is used in the standard proof that Wasserstein distance defines a metric \cite[Chapter 6]{villani-ot-book}. The following result extends the gluing lemma to the setting of relaxed couplings.

\begin{lemma}[Gluing Lemma for $\Cs_\eps$]\label{lem:gluing-Linfinity}
   Fix $\eps = (\eps_1,\eps_2)$ and let
   $(X,\mu_X),(Y,\mu_Y),(Z,\mu_Z)$ be compact probability measure spaces. Let $\pi_1 \in \Cs_{\eps}(\mu_X,\mu_Y)$ and $\pi_2 \in \Cs_{\eps}(\mu_Y,\mu_Z)$. Then there exists a probability measure $\pi_0 \in \P(X\times Y\times Z)$, with $\gamma_1 := (p^{X\times Y})_\sharp \pi_0$ and $\gamma_2:= (p^{Y\times Z})_\sharp\pi_0$, that satisfies the following properties.
   \begin{enumerate}[(A)]
       \item $\gamma_1 \ll \pi_1$, $\pi_1 \ll \gamma_1$, and $\norm{\frac{d\gamma_1}{d\pi_1}}_{L^\infty(\pi_1)}, \norm{\frac{d\pi_1}{d\gamma_1}}_{L^\infty(\gamma_1)}\leq 1+\eps_2$. 
       \item $\gamma_2 \ll \pi_2$, $\pi_2 \ll \gamma_2$, and $\norm{\frac{d\gamma_2}{d\pi_2}}_{L^\infty(\pi_2)}, \norm{\frac{d\pi_2}{d\gamma_2}}_{L^\infty(\gamma_2)}\leq 1+\eps_1$. 
   \end{enumerate}
  
\end{lemma}

\begin{proof}

    By Theorem \ref{thm:distintegration}, we can disintegrate $\pi_i$ into $\set{\pi_i(\cdot|y)}_{y \in Y}$ such that
    \begin{equation}
    \begin{split}
         \pi_1(A\times B) = \int_{B} \pi_1(A|y)d\mu_1(y)
        &\text{ and }
        \pi_2(B\times C) = \int_{B}\pi_2(C|y)d\mu_2(y)
    \end{split}
    \label{eq:glue-disintegrate-pi}
    \end{equation}
    for all $A\subset X$, $B\subset Y$, $C\subset Z$ Borel. Here and below, we use
    \begin{equation}
        \text{$\mu_i$ to denote the $Y$-marginal of $\pi_i$ for $i = 1,2$.}
    \end{equation}

    Let $\pi_0 \in \M_+(X\times Y\times Z)$ be defined by
    \begin{equation}
    \begin{split}
        \pi_0(A\times B\times C) &= \int_B \pi_1(A|y)\pi_2(C|y)d\mu_Y(y) 
    \end{split}
        \label{lem:eq:glue-1}
    \end{equation}
    Plugging $A = X$, $B = Y$, and $C = Z$ into \eqref{lem:eq:glue-1} and using \eqref{eq:glue-disintegrate-pi},  we have
    \begin{equation*}
        \begin{split}
            \pi_0(X\times Y\times Z) &= \int_Y \brac{\int_Xd\pi_1(x|y)}\brac{\int_Zd\pi_2(z|y)}d\mu_Y(y) = 1. 
        \end{split}
    \end{equation*}
    Therefore,
    \begin{equation*}
        \pi_0 \in \P(X\times Y \times Z).
    \end{equation*}
    
    Next, we verify the coupling estimates. Mutual absolute continuity follows directly from these estimates. 
    
    Let $A\subset X$ and $Y\subset Z$ be Borel.  
    \begin{equation*}
        \begin{split}
            \gamma_1 (A\times B) &= 
            \pi_0(A\times B \times Z)
            \\&= 
            \int_B 
            \brac{\int_A d\pi_1(x|y)}
            \brac{\int_Zd\pi_2(z|y)}d\mu_Y(y)
            \\
            &= \int_B \frac{d\mu_Y}{d\mu_1}(y)\int_A d\pi_1(x|y)d\mu_1(y)
            \\
            &\leq 
            \brac{1+\eps_2}
            \int_B\int_A d\pi_1(x|y)d\mu_1(y)\\
            &= (1+\eps_2)\pi_1(A\times B) 
        \end{split}
    \end{equation*}
    On the other hand,
    \begin{equation*}
        \begin{split}
            \pi_1(A\times B)
            &= \int_B \int_A d\pi_1(x|y)d\mu_1(y)
            \\
            &= \int_B \brac{\int_Ad\pi_1(x|y)}\brac{\int_Zd\pi_2(z|y)}d\mu_1(y)
            \\
            &= \int_B \frac{d\mu_1}{d\mu_Y}(y)\brac{\int_A d\pi_1(x|y)}\brac{\int_Zd\pi_2(z|y)}d\mu_Y(y)
            \\
            &\leq 
            \brac{1+\eps_2}
            \int_{B}\int_Ad\pi_1(x|y)\int_Zd\pi_2(z|y)d\mu_Y(y)
            \\
            &\leq (1+\eps_2)\gamma_1(A\times B).
        \end{split}
    \end{equation*}
    Therefore,
    \begin{equation*}
        \norm{\frac{d\gamma_1}{d\pi_1}}_{L^\infty(\pi_1)}, \norm{\frac{d\pi_1}{d\gamma_1}}_{L^\infty(\gamma_1)}\leq 1+\eps_2.
    \end{equation*}
    The estimates in Part (B) are proven similarly. 
\end{proof}

The next theorem is the main result of this subsection, and it roughly states that $\spgw_{\eps,p}$ `almost' has the structure of a semi-metric, in that it satisfies a certain relaxed triangle inequality.

\begin{theorem}[Approximate Triangle Inequality]\label{thm:relaxed triangle}
    Let $\eps = (\eps_1,\eps_2) \in [0,\infty)^2$ and $1 \leq p \leq \infty$. Let $\eps^* := (\eps_1^*, \eps_2^*)$ with 
    \begin{equation*}
        \eps_1^* = \eps_2^* = (1+\eps_1)(1+\eps_2)-1 = \eps_1 + \eps_2 + \eps_1\eps_2.
    \end{equation*}
    Let $X,Y,Z$ be compact pmm-spaces. Then
    \begin{equation}
    \begin{split}
        \spgw_{\eps^*,p}(X,Z ) &\leq (1+\eps_2)^{2/p} \spgw_{\eps,p}(X,Y) 
        + (1+\eps_1)^{2/p}\spgw_{\eps,p}(Y,Z).
    \end{split}
    \label{eq:relaxed-tri-ineq-0}
    \end{equation}
    
\end{theorem}

\begin{proof}
    
    Fix $X,Y,Z$ as above, and let $\pi_1 \in \mathcal{P}(X\times Y)$ and $\pi_2 \in \mathcal{P}(Y\times Z)$ realize $\pgw_{\eps,p}(X,Y)$ and $\pgw_{\eps,p}(Y,Z)$, respectively. Let $\pi_0$ be as in Lemma \ref{lem:gluing-Linfinity} and define
    \begin{equation}
        \pi := (p^{X\times Z})_\sharp\pi_0.
        \label{eq:relaxed-tri-1}
    \end{equation}

    First, we verify that $\pi$ is a feasible coupling.
    Let $\mu_1 := (p^Y)_\sharp\pi_1$.
    For a Borel measurable set $A\subset X$, by formula \eqref{lem:eq:glue-1} and \eqref{eq:relaxed-tri-1}, we have
    \begin{equation*}
        \begin{split}
            (p^X)_\sharp\pi(A) 
            &= \pi_0 (A\times Y \times Z)
            \\
            &= \int_Y \int_{A\times Z}d\pi_1(x|y)d\pi_2(z|y)d\mu_Y(y)
            \\
            &= \int_Y \int_A d\pi_1(x|y)\brac{\int_Zd\pi_2(z|y)}d\mu_Y(y)
            \\
            &= \int_Y  \frac{d\mu_Y}{d\mu_1}(y)\int_A d\pi_1(x|y)d\mu_Y(y)\\
            &\leq \brac{1+\eps_2}\int_Y\int_A d\pi_1(x|y)d\mu_1(y)
            \\
            &\leq (1+\eps_2)\int_{A\times Y} d\pi_1(x,y)
            \\
            &\leq (1+\eps_1)(1+\eps_2) \mu_X(A).
        \end{split}
        \label{eq:relaxed-tri-ineq-1}
    \end{equation*}
    On the other hand,
    \begin{equation*}
        \begin{split}
            \mu_X(A)
            &\leq (1+\eps_1)\pi_1(A,Y)
            \\
            &= (1+\eps_1)\int_Y \int_A d\pi_1(x|y)d\mu_1(y)
            \\
            &= (1+\eps_1)\int_Y \frac{d\mu_1}{d\mu_Y}(y)\int_Ad\pi_1(x|y)\brac{\int_Z d\pi_2(z|y)}d\mu_Y(y) 
            \\
            &\leq
            (1+\eps_1)\brac{1+\eps_2}\int_{Y}\int_A\int_Z d\pi_1(x|y)d\pi_2(z|y)d\mu_Y(y)\\
            &\leq (1+\eps_1)(1+\eps_2)\pi_0(A\times Y\times Z)
            \\
            &= (1+\eps_1)(1+\eps_2)(p^X)_\sharp \pi(A).
        \end{split}
    \end{equation*}
    The estimate for $(p^Z)_\sharp\pi$ is similar.
    Therefore,
    \begin{equation}
        \pi \in \Cs_{\eps^*}(\mu_X,\mu_Z)
        \label{eq:relaxed-tri-ineq-3}.
    \end{equation}

    By the triangle inequality for $L^p$, we have
    \begin{equation}
        \begin{split}
        \spgw_{\eps^*,p}(X,Z)
        \leq
            \norm{d_X - d_Z}_{L^p(\pi\otimes \pi)}
        &= \norm{d_X - d_Z}_{L^p(\pi_0\otimes \pi_0)}
        \\
        &\leq \norm{d_X - d_Y}_{L^p(\pi_0\otimes \pi_0)} + \norm{d_Y - d_Z}_{L^p(\pi_0\otimes \pi_0)}.
        \end{split}
        \label{eq:relaxed-tri-ineq-4}
    \end{equation}
    Suppose $1 \leq p < \infty$. Since $\pi_1$ and $\pi_2$ realize  $\spgw_{\eps,p}(X,Y)$ and $\spgw_{\eps,p}(Y,Z)$, respectively, we can use Lemma \ref{lem:gluing-Linfinity} to continue estimating \eqref{eq:relaxed-tri-ineq-4}:
    \begin{equation}
        \begin{split}
            &\leq (1+\eps_2)^{2/p}\norm{d_X - d_Y}_{L^p(\pi_1\otimes \pi_1)} + (1+\eps_1)^{2/p}\norm{d_Y -d_Z}_{L^p(\pi_2\otimes \pi_2)} \\
            &\leq 
            (1+\eps_2)^{2/p}\spgw_{\eps,p}(X,Y) + (1+\eps_1)^{2/p}\spgw_{\eps,p}(Y,Z)
        \end{split}
        \label{eq:relaxed-tri-ineq-5}
    \end{equation}
    Suppose $p = \infty$ (so $(1+\eps_i)^{2/p} = 1$). By Lemma \ref{lem:gluing-Linfinity}, $(p^{X\times Y})_\sharp\pi_0$ and $\pi_1$ are absolutely continuous with respect to each other, so are $(p^{Y\times Z})_\sharp\pi_0$ and $\pi_2$. Therefore, we may continue estimating \eqref{eq:relaxed-tri-ineq-4}:
    \begin{equation}
        \begin{split}
            &= \norm{d_X - d_Y}_{L^p(\pi_1\otimes \pi_1)} + \norm{d_Y -d_Z}_{L^p(\pi_2\otimes \pi_2)} \leq 
            \spgw_{\eps,p}(X,Y) + \spgw_{\eps,p}(Y,Z)
        \end{split}
        \label{eq:relaxed-tri-ineq-6}.
    \end{equation}
    The proof of \eqref{eq:relaxed-tri-ineq-0} is complete in view of \eqref{eq:relaxed-tri-ineq-3}--\eqref{eq:relaxed-tri-ineq-6}. 

\end{proof}

\begin{remark}
In the formally limiting case $\eps = (0,\infty)$, we have $\eps^* = (\infty, \infty)$, and the left-hand side of \eqref{eq:relaxed-tri-ineq-0} is simply zero and the bound is trivial. However, as mentioned in Remark \ref{rem:special cases}(B), $\spgw_{(0,\infty), p}$ is a metric if and only if $p = \infty$. See \cite{clark2024generalized}, where the proof of the triangle inequality avoids using any tool like the Gluing Lemma (Lemma \ref{lem:gluing-Linfinity}).
\end{remark}

\subsection{Related Work}\label{sect: related works}

As was mentioned above, the partial GW distances that we study in this paper are related to others that have appeared in the literature; in particular, they generalize the distances studied in \cite{chapel2020partial}. In this section, we precisely contextualize our framework within the literature, using the properties of the partial Gromov-Wasserstein distances that we have established thus far.

\subsubsection{Equivalence to Chapel et al.}
\label{sect: relation to chapel}

A notion of distance between mm-spaces whose underlying sets are finite, and whose measures are not necessarily of the same total mass is defined by Chapel et al.\ in \cite[Section 3.2]{chapel2020partial} (based on earlier ideas of Caffarelli-McCann~\cite{caffarelli2010free} and Figalli~\cite{figalli2010optimal}). The following definition generalizes this idea to spaces that are not necessarily finite and extends the $L^2$-type construction of Chapel et al.\ to a family of $L^p$-type distances. We denote the resulting distances as $\pgwt_{\delta,p}$, where the prepended $\mathsf{m}$ indicates a `mass constraint' and that the distance is valid for general mm-spaces, rather than only pmm-spaces.

\begin{definition}
\label{def:Chapel}
Let $\Tilde{\mu}_X \in \mathcal{M}_+(X)$ and $\Tilde{\mu}_Y \in \mathcal{M}_+(Y)$ be (not necessarily probability) measures with total mass $m_X := |\Tilde{\mu}_X|$ and $m_Y := |\Tilde{\mu}_Y|$ respectively.
Given mm-spaces $\Tilde{X} = (X,d_X,\Tilde{\mu}_X) $ and $ \Tilde{Y} = (Y,d_Y,\Tilde{\mu}_Y)$, $0\leq \delta \leq \min\{m_X, m_Y\}$, and $p \in [1,\infty]$, define
\begin{equation}\label{eq:chapel_PGW}
    \pgwt_{\delta,p}(\Tilde{X},\Tilde{Y}) := \inf_{\Tilde{\pi} \in \Tilde{\C}_{\delta}(\Tilde{\mu}_X,\Tilde{\mu}_Y) }\norm{d_X - d_Y}_{L^p(\Tilde{\pi}\otimes \Tilde{\pi})} ,
\end{equation}
where
\begin{equation}
    \Tilde{\C}_{\delta}(\Tilde{\mu}_X,\Tilde{\mu}_Y) := \set{\Tilde{\pi} \in \mathcal{M}_+(X\times Y):
   \Tilde{\pi}_X \leq \Tilde{\mu}_X,\, 
    \Tilde{\pi}_Y \leq \Tilde{\mu}_Y,\, 
    \abs{\Tilde{\pi}} = \delta 
    }.
    \label{eq:coupling-chapel}
\end{equation}
As before, $\Tilde{\pi}_X$ and $\Tilde{\pi}_Y$ are the $X$- and $Y$- marginals of $\Tilde{\pi}$, respectively.
\end{definition}

Intuitively, we can compare the approach of the current article with that of \cite{chapel2020partial} as follows. We can interpret
$\pgwt_{\delta,p}(\Tilde{X},\Tilde{Y})$ as moving only a $\delta$-portion of the total masses of $X$ and $Y$ via $\Tilde{\pi}$, but the portions of moved mass must be restrictions of the original measures $\Tilde{\mu}_X$ and $\Tilde{\mu}_Y$. On the other hand, $\pgw_{\eps,p}(X,Y)$ allows for an $\eps$ amount of rearrangement of the masses of the original measures $\mu_X$ and $\mu_Y$, but the full amount of rearranged masses must be transported.   Despite the difference in their formulations,  we now show that the two partial Gromov-Wasserstein distances are equivalent up to scaling.

\begin{proposition}
\label{prop:equivalence}

The following hold, with the convention $\frac{0}{0} = 0$.
\begin{enumerate}[(A)]
    \item Given mm-spaces $\Tilde{X} = (X,d_X, \tilde\mu_X)$, $\Tilde{Y} = (X,d_Y, \tilde\mu_Y)$ and $0\leq \delta \leq \min\{m_X, m_Y\}$, we have
\[
\delta^{-1}\pgwt_{\delta,p}(\Tilde{X},\Tilde{Y}) = \pgw_{\eps,p}(X,Y) ,
\]
where $1+ \eps_1 = \frac{m_X}{\delta}$, $1+ \eps_2 = \frac{m_Y}{\delta}$, and $X = (X,d_X,\frac{\Tilde{\mu}_X}{m_X}), Y = (Y,d_Y,\frac{\Tilde{\mu}_Y}{m_Y})$ are pmm-spaces.

\item 
Similarly, given pmm-spaces $X = (X,d_X,\mu_X),Y = (Y,d_Y,\mu_Y), \eps_1\geq0$ and $\eps_2\geq0$, we have for any $\delta> 0$
\[
\pgw_{\eps,p}(X,Y) = \delta^{-1}\pgwt_{\delta,p}(\Tilde{X},\Tilde{Y}) 
\]
where $\Tilde{X} = (X,d_X, (1+\eps_1) \delta \cdot\mu_X)$ and $\Tilde{Y} = (Y,d_Y, (1+\eps_2) \delta \cdot\mu_Y)$.
\end{enumerate}
\end{proposition}

This yields the following corollary immediately.

\begin{corollary}\label{cor:equivalence}
    Given pmm-spaces $X$ and $Y$ we have 
    \[
    \delta^{-1}\pgwt_{\delta,p}(X,Y) = \pgw_{\eps,p}(X,Y) \quad \mbox{when} \quad  1+ \eps_1 = 1+ \eps_2 =\frac{1}{\delta}.
    \]
In other words, the partial Gromov--Wasserstein distance $\pgw_{\eps,p}$ is equivalent to the partial GW distance $\pgwt_{\delta,p}$ of~\cite{chapel2020partial} on the space of pmm-spaces.
\end{corollary}

\begin{remark}
    In light of this corollary, one can use either formulation, $\pgw$ or $\pgwt$, when deriving theoretical properties. The motivation for introducing and working primarily with the $\pgw$ variant is that we find it more natural to work exclusively with pmm-spaces, which is a convenient setting for, e.g., establishing our version of the Gluing Lemma (Lemma \ref{lem:gluing-Linfinity}). 
\end{remark}

\begin{proof}[Proof of Proposition \ref{prop:equivalence}]

We will show (A). The proof of (B) is similar. 

If $\delta = 0$ then $\pgwt_{\delta,p}(\Tilde{X},\Tilde{Y}) = 0$, and $\delta^{-1}\pgwt_{\delta,p}(\Tilde{X},\Tilde{Y}) = 0$ according to our convention. On the other hand, $\pgw_{(\infty,\infty),p}(X,Y) = 0$ since we can concentrate all the mass at a single point for both $X$ and $Y$. The equality holds. 

For the rest of the proof, we assume that $\delta > 0$.

For any $\Tilde{\pi} \in \Tilde{\C}_{\delta}(\Tilde{\mu}_X,\Tilde{\mu}_Y)$, set $\pi := \frac{\Tilde{\pi}}{\delta}$ so that $\pi$ is a probability measure i.e., $|\pi| = 1$.
Let $1 + \eps_1 = \frac{m_X}{\delta}$ and $1 + \eps_2 = \frac{m_Y}{\delta}$. We have
\[
(p^X)_\sharp\pi \leq \frac{\Tilde{\mu}_X}{\delta} = (1+\eps_1) \frac{\Tilde{\mu}_X}{m_X},
\quad
\text{and}
\quad
(p^Y)_\sharp\pi \leq \frac{\Tilde{\mu}_Y}{\delta} = (1+\eps_2) \frac{\Tilde{\mu}_Y}{m_Y}.
\]
Since $|\Tilde{\mu}_X| = m_X$ and $|\Tilde{\mu}_Y| = m_Y$, $\frac{\Tilde{\mu}_X}{m_X}$ and $\frac{\Tilde{\mu}_Y}{m_Y}$ are probability measures, so $\pi \in \C_\eps\big(\frac{\Tilde{\mu}_X}{m_X},\frac{\Tilde{\mu}_Y}{m_Y}\big)$.
Similarly, for any $\pi \in \C_\eps\big(\frac{\Tilde{\mu}_X}{m_X},\frac{\Tilde{\mu}_Y}{m_Y}\big)$ with $\eps = ( \frac{m_X}{\delta} - 1,\frac{m_Y}{\delta} - 1)$, we have $\delta \pi \in \Tilde{\C}_{\delta}(\Tilde{\mu}_X,\Tilde{\mu}_Y)$.
Therefore, 
\[
\frac{1}{\delta} \cdot \Tilde{\C}_{\delta}(\Tilde{\mu}_X,\Tilde{\mu}_Y) = \C_\eps\bigg(\frac{\Tilde{\mu}_X}{m_X},\frac{\Tilde{\mu}_Y}{m_Y}\bigg).
\]
As a result,
\begin{align*}
\delta^{-1}\pgwt_{\delta,p}(\Tilde{X},\Tilde{Y})
&= \delta^{-1} \cdot \inf_{\Tilde{\pi} \in \Tilde{\C}_{\delta}(\Tilde{\mu}_X,\Tilde{\mu}_Y) }\norm{d_X - d_Y}_{L^p(\Tilde{\pi}\otimes \Tilde{\pi})} \\
&=\inf_{\frac{\Tilde{\pi}}{\delta} \in \frac{1}{\delta} \cdot  \Tilde{\C}_{\delta}(\Tilde{\mu}_X,\Tilde{\mu}_Y) }\norm{d_X - d_Y}_{L^p(\frac{\Tilde{\pi}}{\delta}\otimes \frac{\Tilde{\pi}}{\delta})} \\
&=\inf_{\pi \in\C_\eps\big(\frac{\Tilde{\mu}_X}{m_X},\frac{\Tilde{\mu}_Y}{m_Y}\big)}\norm{d_X - d_Y}_{L^p(\pi\otimes \pi)} \\
&=  \pgw_{\eps,p}(X,Y),
\end{align*}
where $X = (X,dx,\frac{\Tilde{\mu}_X}{m_X})$ and $Y = (Y,dy,\frac{\Tilde{\mu}_Y}{m_Y})$.
\end{proof}

\subsubsection{Lagrangian form, relaxed couplings, and relation to Bai et al.} 
\label{sect: relation to bai}
We compare our construction with the partial GW distances defined in Bai et al.\ \cite{Vanderbilt2024-PGW} (which are, in turn, closely related to the \textit{unbalanced GW distances} introduced in \cite{sejourne2021-conic}). Recall Definition \ref{def:partial_GW}. Let $(X,d_X,\mu_X)$ and $(Y,d_Y,\mu_Y)$ be two pmm-spaces. The optimization problem associated with $\pgw_{\eps,p}(X,Y)$ has the form
\begin{equation}
    \begin{split}
        \text{minimize } f(\pi) \quad \text{ subject to } \quad  g_i(\pi)\leq 0
        \text{ for } i = 1,2,
    \end{split}
    \label{eq:nonlinear-program}
\end{equation}
where $f(\pi) = \norm{d_X - d_Y}_{L^p(\pi\otimes \pi)}^p$, $g_1(\pi) = \norm{\frac{d\mu_X}{d\pi_X}-1}_{L^\infty(\pi_X)} - \eps_1$, and $g_2(\pi) = \norm{\frac{d\mu_Y}{d\pi_Y}-1}_{L^\infty(\pi_Y)} - \eps_2$.

Let $ L: \P(X\times Y) \times [0,\infty)^2 \to \R$ be the Lagrangian associated with  \eqref{eq:nonlinear-program}, given by
\begin{equation*}
   L(\pi,\lambda_1,\lambda_2) = f(\pi) + \sum_{i = 1}^2 \lambda_i g_i(\pi).
\end{equation*}
A standard argument yields
\begin{equation*}
    \pgw_{\eps,p}(X,Y) = \inf_{\pi\in \P(X\times Y)} \sup_{\lambda_1,\lambda_2 \geq 0} L(\pi,\lambda_1,\lambda_2) \geq \sup_{\lambda_1,\lambda_2 \geq 0}\inf_{\pi\in \P(X\times Y)}  L(\pi,\lambda_1,\lambda_2)
\end{equation*}
Consider the Lagrangian dual of $f$ given by
\begin{equation*}
    \begin{split}
        f^*(\lambda_1,\lambda_2) &= \inf_{\pi \in \P(X\times Y)} L(\pi,\lambda_1, \lambda_2)
        \\
        &= \inf_{\pi \in \P(X\times Y)}\brac{
        \norm{d_X - d_Y}^p_{L^p(\pi\otimes \pi)} + \lambda_1 \norm{\frac{d\mu_X}{d\pi_X}-1}_{L^\infty(\pi_X)} + \lambda_2 \norm{\frac{d\mu_Y}{d\pi_Y}-1}_{L^\infty(\pi_Y)}
        } \\
        &\hspace{11cm}
        - \sum_{i = 1}^2 \lambda_i \eps_i.
    \end{split}
\end{equation*}
Note that $\norm{\frac{d\mu_X}{d\pi_X}-1}_{L^\infty(\pi_X)}$ and $\norm{\frac{d\mu_Y}{d\pi_Y}-1}_{L^\infty(\pi_Y)}$ can be written as limits of `$\chi^q$-divergences':
\begin{equation*}
    \begin{split}
        \HD_{\infty} (\mu_X,\pi_X)
        :=\norm{\frac{d\mu_X}{d\pi_X}-1}_{L^\infty(\pi_X)} &= \lim_{q \to \infty}
        \brac{\int_X \abs{\frac{d\mu_X}{d\pi_X}-1}^q d\pi_X}^{1/q},\\
        \HD_{\infty} (\mu_Y,\pi_Y):= \norm{\frac{d\mu_Y}{d\pi_Y}-1}_{L^\infty(\pi_Y)} &= \lim_{q \to \infty}
        \brac{\int_Y \abs{\frac{d\mu_Y}{d\pi_Y}-1}^q d\pi_Y}^{1/q}.
    \end{split}
\end{equation*}
Using this notation, the optimization problem associated with $f^*(\lambda_1,\lambda_2)$ can be written as
\begin{equation*}
    \text{Minimize }
    \norm{d_X - d_Y}_{L^p(\pi\otimes \pi)}^p + \lambda_1 \HD_\infty(\mu_X,\pi_X) + \lambda_2 \HD_\infty(\mu_Y,\pi_Y) + \sum_{i  = 1}^2 \lambda_i \eps_i
\end{equation*}
which resembles the Lagrangian form in \cite{Vanderbilt2024-PGW}, where the authors penalize dissimilarity in terms of the total variation, i.e., a `$\chi^1$-divergence'. We see that the optimization problem associated with $\pgw_{\eps,p}(X,Y)$ can be interpreted as a primal problem with a dual relaxation closely related to the distance defined in Bai et al. \cite{Vanderbilt2024-PGW}.

\subsubsection{Relation to Kong et al.}\label{sect: relation to kong} 
Another notion of relaxed GW distance, dubbed \textit{outlier-robust Gromov-Wasserstein distance}, is considered by Kong et al.\ in \cite{kong2024outlier}. Given a collection of positive, real hyperparameters $\rho = (\eps_1,\eps_2,\rho_1,\rho_2)$, the associated outlier-robust Gromov-Wasserstein distance between pmm-spaces $X$ and $Y$ is defined to be 
\begin{equation}\label{eqn:outlier_robust_gw}
\begin{split}
\mathsf{RGW}_\rho(X,Y) &\coloneqq \inf_{\alpha \in \mathcal{P}(X), \beta \in \mathcal{P}(Y)} \inf_{\pi \in \mathcal{M}_+(X \times Y)} \mathsf{dis}_2(\pi)^2 + \rho_1 \mathsf{KL}(\pi_X \mid \alpha) + \rho_2 \mathsf{KL}(\pi_Y \mid \beta) \\
&\hspace{2in} \mbox{sub. to.} \quad  \mathsf{KL}(\mu_X|\alpha) \leq \eps_1, \quad \mathsf{KL}(\mu_Y \mid \beta) \leq \eps_2,
\end{split}
\end{equation}
where $\mathsf{KL}$ denotes Kullback-Leibler divergence and $\mathsf{dis}_2$ is the same distortion function from Definition \ref{def:distortion}. Taking $\rho_1,\rho_2 \to \infty$, this becomes
\begin{equation}\label{eqn:outlier_robust_gw_tau_infty}
\begin{split}
\mathsf{RGW}_\rho(X,Y) &\coloneqq \inf_{\alpha \in \mathcal{P}(X), \beta \in \mathcal{P}(Y)} \inf_{\pi \in \mathcal{C}(\alpha,\beta)} \mathsf{dis}_2(\pi)^2  \\
&\hspace{1in} \mbox{sub. to.} \quad  \mathsf{KL}(\mu_X|\alpha) \leq \eps_1, \quad \mathsf{KL}(\mu_Y \mid \beta) \leq \eps_2.
\end{split}
\end{equation}

We now observe that this problem is closely related to the $\pgw$ and $\spgw$ problems. Indeed, in light of Theorem \ref{thm:nondeg-isomorphism}, the $\spgw_{\eps,2}$ problem can be expressed as 
\begin{equation}\label{eqn:spgw_related_to_kong}
\begin{split}
    \spgw_{\eps,2}(X,Y) &= \inf_{\alpha \in \mathcal{P}(X),\beta \in \mathcal{P}(Y)} \inf_{\pi \in \mathcal{C}(\alpha,\beta)} \mathsf{dis}_2(\pi) \\
    &\hspace{1in} \mbox{sub. to.} \quad \left\|\frac{d\alpha}{d\mu_X} - 1\right\|_{L^\infty(\mu_X)}, \left\|\frac{d\mu_X}{d\alpha} - 1\right\|_{L^\infty(\alpha)} \leq \eps_1 \\ 
    &\hspace{1.65in} \left\|\frac{d\beta}{d\mu_Y} - 1\right\|_{L^\infty(\mu_Y)}, \left\|\frac{d\mu_Y}{d\beta} - 1\right\|_{L^\infty(\beta)} \leq \eps_2.
\end{split}
\end{equation}
Observe that the only difference between the problems \eqref{eqn:outlier_robust_gw_tau_infty} and \eqref{eqn:spgw_related_to_kong} is the manner in which the measures $\alpha$ and $\beta$ are constrained to be `close' to the targets $\mu_X$ and $\mu_Y$ (and the fact that distortion is squared in \eqref{eqn:outlier_robust_gw_tau_infty}, with is inconsequential). 

We note that the work of Kong et al.\ does not address any metric properties related to the optimization problem \ref{eqn:outlier_robust_gw}. Based on its connection to the symmetrized partial Gromov-Wasserstein framework, and our results above, we conjecture that \ref{eqn:outlier_robust_gw} also fails to induce a true metric structure on the space of pmm-spaces.

\section{Robust Partial Gromov-Wasserstein Distance}\label{sec:robust_pGW}

To define a notion of partial Gromov-Wasserstein distance with more well-behaved metric properties, we introduce a construction similar to those appearing in the definitions of the Prokhorov, Ky Fan, and Robust Wasserstein metrics \cite{RSZ2024, prokhorov1956leviprokh, billingsley2013convergence}. 

\begin{convention}
    For the rest of the paper, we will work with $\pgw_{\eps,p}$ with $\eps_1 = \eps_2$, so we treat $\eps$ as a scalar.
\end{convention}

\begin{definition}[Robust Partial Gromov--Wasserstein Distance] 
For each $k \geq 0$ and $p\geq 1$, we define the associated \define{robust partial Gromov-Wasserstein distance} by
\begin{equation}
    \pgw_{p}^k(X,Y) := \inf\set{
    \frac{\eps}{1+\eps} \geq 0 : \pgw_{\eps,p}(X,Y) \leq k\eps
    }.
\end{equation}
\end{definition}

The additional parameter $k$ allows for data-dependent fine-tuning of the metric. We will see below that the choice of $k$ does not make any qualitative difference regarding the properties of $\pgw_{p}^k$. The results in this section show that $\pgw_p^k$ defines a metric with theoretical robustness guarantees.  

\subsection{Metric Properties of \texorpdfstring{$\pgw^k_p$}{PGWpk}}

The robust partial GW distances are obviously symmetric. We will now show that they are non-degenerate and satisfy the triangle inequality.

\paragraph{Non-degeneracy.}
Recall that two mm-spaces are isomorphic if there exists a measure-preserving isometric bijection between them.

\begin{theorem}[Non-Degeneracy of $\pgw^k_p$]\label{thm:nondeg-isomorphism}
    Suppose $k > 0$ and let $X$ and $Y$ be compact, fully supported pmm-spaces. Then $\pgw_{p}^k(X,Y) = 0$ if and only if $X$ and $Y$ are isomorphic.
\end{theorem}

\begin{proof}
   Suppose $X$ and $Y$ are isomorphic, then $\gw_p(X,Y) = 0$, and we have  $\pgw_{\eps,p}(X,Y) \leq \gw_p(X,Y) = 0$.
    
    Suppose $\pgw_{p}^k(X,Y) = 0$.
Let $\frac{\eps_j}{1+\eps_j} \searrow 0$ (so $\eps_j \to 0$) along with a sequence of measures $\pi_j \in \C_{\eps_j}(\mu_X,\mu_Y)$ with
    \begin{equation}
        \norm{d_X - d_Y}_{L^p(\pi_j\otimes \pi_j)}\leq 2k\eps_j.
    \end{equation}
   By Prokhorov's Theorem, we can replace $(\pi_j)$ by a subsequence, still called $(\pi_j)$, such that $\pi_j \to \pi \in \P(X\times Y)$ weakly.
     Since $\eps_j \to 0$, we have 
    \begin{equation}
        (p^X)_\sharp\pi \leq \mu_X \text{ and } (p^Y)_\sharp\pi\leq \mu_Y
        \label{eq:thm:iso-1}
    \end{equation}
    Since $\pi,\mu_X,\mu_Y$ are probability measures, the marginal condition \eqref{eq:thm:iso-1} forces
    \begin{equation}
        (p^X)_\sharp\pi = \mu_X \text{ and } (p^Y)_\sharp\pi= \mu_Y.
    \end{equation}
    On the other hand, by weak convergence (or Lemma \ref{lem:distortion-continuous}(B) in the case $p = \infty$), we have
    \begin{equation}
        \norm{d_X - d_Y}_{L^p(\pi\otimes \pi)} = 0.
    \end{equation}
    Namely, $\gw_p(X,Y) = 0$. Therefore, $X$ and $Y$ are isomorphic.
\end{proof}

\paragraph{Triangle Inequality.}
Next, we show that $\pgw_{p}^k$ distance satisfies the triangle inequality.
The proof of triangle inequality makes use of the equivalence between $\pgw$ and $\pgwt$, as the analysis is simpler for the latter version of the distance. We therefore first define a robust version for $\pgwt$.

\begin{definition}[Robust $\pgwt$]
Let $X = (X,d_X,\mu_X)$ and $Y =(Y,d_Y,\mu_Y)$ be pmm-spaces.
Consider $\pgwt_{\delta,p}(X,Y)$ from Definition~\ref{def:Chapel}.
For any $k\geq0$, define 
\begin{equation}
    \pgwt_{p}^{k}(X,Y) := \inf\set{
    0\leq \delta \leq 1 : \pgwt_{1-\delta,p}(X,Y) \leq k\delta
    }.
\end{equation}   
\end{definition}

\begin{theorem}[Triangle Inequality for $\pgwt_p^k$]
\label{thm: chapel triangle inequality}
For any $1 \leq p\leq \infty$, $k\geq 0$, and pmm-spaces $X,Y,Z$, we have
\begin{equation}
    \pgwt_{p}^{k}(X,Z) \leq \pgwt_{p}^{k}(X,Y) + \pgwt_{p}^{k}(Y,Z).
\end{equation}

\end{theorem}

\begin{proof}
To simplify notation, let 
\[
\Bar{\delta}_1 = \pgwt_p^k(X,Y) 
\quad \mbox{and} \quad  
\Bar{\delta}_2 = \pgwt_p^k(Y,Z).
\]
Suppose $\Bar{\delta}_1 + \Bar{\delta}_2 \geq 1$, then triangle inequality holds trivially, since
$\pgwt_p^k(X,Z) \leq 1 \leq \Bar{\delta}_1 + \Bar{\delta}_2$.

For the rest of the proof, assume
\begin{equation*}
    \Bar{\delta}_1 + \Bar{\delta}_2 < 1.
\end{equation*}

By definition, for any $\delta_1>\Bar{\delta}_1$ and $\delta_2>\Bar{\delta}_2$, we have
\[
\pgwt_{1-\delta_1,p}(X,Y) \leq k\delta_1 
\quad \mbox{and} \quad   
\pgwt_{1-\delta_2,p}(Y,Z) \leq k\delta_2.
\]
Fix arbitrary $\delta_1>\Bar{\delta}_1$ and $\delta_2>\Bar{\delta}_2$ with $\delta_1+\delta_2 \leq 1$.
Suppose that $\pgwt_{1-\delta_1,p}(X,Y)$ and $\pgwt_{1-\delta_2,p}(Y,Z)$ are realized by $\tilde{\pi}_1$ and $\tilde{\pi}_2$, respectively.
Notice that 
\begin{equation}
    (p^X)_\sharp\Tilde{\pi}_1 \leq \mu_X,\,
    (p^Y)_\sharp\Tilde{\pi}_1 \leq \mu_Y,\,
    \text{ and }
    |\tilde{\pi}_1| = 1-\delta_1.
    \label{eq:thm: chapel triangle-1}
\end{equation}
Intuitively, $\tilde{\pi}_1$ corresponds to a transport plan moving a total mass $1-\delta_1$.
Similarly, 
\begin{equation}
    (p^Y)_\sharp\Tilde{\pi}_2 \leq \mu_Y,\,
    (p^Z)_\sharp\Tilde{\pi}_2 \leq \mu_Z,\,\text{ and }
    |\tilde{\pi}_2| = 1-\delta_2.
    \label{eq:thm: chapel triangle-2}
\end{equation}
For $i = 1, 2$, set
\begin{equation*}
    \tilde{\mu}_{Y,i}:= (p^Y)_\sharp\Tilde{\pi}_i.
\end{equation*}
Then for $i = 1, 2$, 
\begin{equation*}
    |\tilde{\mu}_{Y,i}| = 1-\delta_i
\end{equation*}
We also set
\begin{equation*}
    \tilde{\mu}^c_{Y} := \tilde{\mu}_{Y,1} \wedge \tilde{\mu}_{Y,2}
\end{equation*}
where $\mu\wedge\nu(S) := \inf\set{
\mu(S\cap B) + \nu(S\setminus B) : B \text{ measurable}
}$. Intuitively, $\tilde{\mu}^c_{Y}$ registers the common mass transported by both $\tilde\pi_1$ and $\tilde\pi_2$, hence the superscript $c$.

Set $\delta_c := |\tilde{\mu}^c_{Y}| $.
Since $\tilde{\mu}^c_{Y}, \tilde{\mu}_{Y,1}, \tilde{\mu}_{Y,2} \leq \mu_Y$, we have
\begin{equation*}
    \mu_Y - \tilde{\mu}^c_{Y} = (\mu_Y - \tilde{\mu}_{Y,1}) \vee (\mu_Y - \tilde{\mu}_{Y,2})\leq (\mu_Y - \tilde{\mu}_{Y,1}) + (\mu_Y - \tilde{\mu}_{Y,2}).
\end{equation*}
Here, $\mu\vee\nu(S) := \sup\set{
\mu(S\cap B) + \nu(S\setminus B) : B \text{ measurable})
}$
Therefore, $|\mu_Y - \tilde{\mu}^c_{Y}| \leq \delta_1 + \delta_2$ and $|\tilde{\mu}^c_{Y}| = \delta_c \geq 1 - \delta_1 - \delta_2$.

Using the notation of Theorem \ref{thm:distintegration}, let $\set{\tilde\pi_1(\cdot|y)}_{y \in Y}$ and $\set{\tilde{\pi}_2(\cdot|y)}_{y \in Y}$ be the disintegration of $\tilde\pi_1$ and $\tilde\pi_2$. Define
\begin{equation*}
    d\tilde\pi_1^c(x,y) := d\tilde\pi_1(x|y)d\tilde\mu_Y^c(y)
    \quad\text{ and }\quad
    d\tilde\pi_2^c(y,z) := d\tilde\pi_2(z|y)d\tilde\mu_Y^c(y),
\end{equation*}
so that for $i = 1, 2$, 
\begin{equation*}
    \tilde\pi_i^c \leq \tilde\pi_i,\, (p^Y)_{\sharp}\tilde\pi_1^c = \tilde\mu_Y^c,\, \text{ and } \abs{\tilde\pi_i^c} = \delta_c.
\end{equation*}
We also set
\begin{equation}
    \tilde{\mu}^c_{X}:= (p^X)_\sharp\Tilde{\pi}^c_1
    \quad\text{ and }
    \quad
    \tilde{\mu}^c_{Z}:= (p^Z)_\sharp\Tilde{\pi}^c_2.
    \label{eq:thm: chapel triangle-3}
\end{equation}
In view of \eqref{eq:thm: chapel triangle-1}--\eqref{eq:thm: chapel triangle-3}, we have
\begin{equation}
    \tilde{\mu}^c_{X} \leq \mu_X
    \text{ and }
    \tilde{\mu}^c_{Z} \leq \mu_Z.
    \label{eq:thm: chapel triangle-4}
\end{equation}

We are ready to show $\pgwt_{1-\delta_1- \delta_2,p}(X,Z)\leq k\delta_1 + k\delta_2$.
Suppose $\inf_{\tilde{\pi} \in \Tilde{\C}_{\delta_c}(\mu_X,\mu_Z) }\norm{d_X - d_Z}_{L^p(\tilde{\pi}\otimes \tilde{\pi})}$ is achieved by $\tilde{\pi}^\star  \in \Tilde{\C}_{\delta_c}(\mu_X,\mu_Z)$.
Since $\frac{1-\delta_1 - \delta_2}{\delta_c}\leq 1$, we have
\begin{equation*}
    \frac{1-\delta_1 - \delta_2}{\delta_c} \cdot \tilde{\pi}^\star \in \Tilde{\C}^{1-\delta_1 - \delta_2}(\mu_X,\mu_Z).
\end{equation*}
By the definition of $\Tilde{\C}_{\delta_c}$ in \eqref{eq:coupling-chapel},
\begin{equation}
    \{\tilde{\pi} \in \mathcal{M}_+(X\times Z):\, (p^X)_\sharp\Tilde{\pi} = \tilde{\mu}^c_{X},\, (p^Z)_\sharp\Tilde{\pi} = \tilde{\mu}^c_{Z},\, \abs{\tilde{\pi}} = \delta_c \} \subseteq \Tilde{\C}_{\delta_c}(\mu_X,\mu_Z)
    \label{eq:thm: chapel triangle-5}
\end{equation}
As a consequence, 
\begin{align*}
\pgwt_{1-\delta_1-\delta_2,p}(X,Z)
&= \inf_{\tilde{\pi} \in \Tilde{\C}^{1-\delta_1-\delta_2}(\mu_X,\mu_Z) }\norm{d_X - d_Z}_{L^p(\tilde{\pi}\otimes \tilde{\pi})} \\
&\leq \norm{d_X - d_Z}_{L^p(\frac{1-\delta_1 - \delta_2}{\delta_c} \cdot \tilde{\pi}^\star\otimes \frac{1-\delta_1 - \delta_2}{\delta_c} \cdot \tilde{\pi}^\star)} \\
&\leq \norm{d_X - d_Z}_{L^p(\tilde{\pi}^\star\otimes \tilde{\pi}^\star)} \\
&=\inf_{\tilde{\pi} \in \Tilde{\C}_{\delta_c}(\mu_X,\mu_Z) }\norm{d_X - d_Z}_{L^p(\tilde{\pi}\otimes \tilde{\pi})} \\
\text{\eqref{eq:thm: chapel triangle-5}}\quad&\leq \inf_{\set{\tilde{\pi} \in \mathcal{M}_+(X\times Z):\, (p^X)_\sharp\Tilde{\pi} = \tilde{\mu}^c_{X},\, (p^Z)_\sharp\Tilde{\pi} = \tilde{\mu}^c_{Z},\, \abs{\tilde{\pi}} = \delta_c }} \norm{d_X - d_Z}_{L^p(\tilde{\pi}\otimes \tilde{\pi})}.
\end{align*}
Notice that
\[
\inf_{\set{\tilde{\pi} \in \mathcal{M}_+(X\times Z):\, \tilde{\pi}_X = \tilde{\mu}^c_{X},\, \tilde{\pi}_Z = \tilde{\mu}^c_{Z},\, \abs{\tilde{\pi}} = \delta_c }} \norm{d_X - d_Z}_{L^p(\tilde{\pi}\otimes \tilde{\pi})} = \delta_c \cdot \gw_p(X^c,Z^c),
\]
where $X^c := (X,d_X,\frac{\tilde{\mu}^c_{X}}{\delta_c})$ and $Z^c := (Z,d_Z,\frac{\tilde{\mu}^c_{Z}}{\delta_c})$ are pmm-spaces. Consider an auxiliary pmm-space $Y^c := (Y,d_Y,\frac{\tilde{\mu}^c_{Y}}{\delta_c})$.
By the triangle inequality for $\gw_p$, we have
\[
\delta_c \cdot \gw_p(X^c,Z^c)
\leq \delta_c \cdot \gw_p(X^c,Y^c) + \delta_c \cdot \gw_p(Y^c,Z^c).
\]
Therefore,
\begin{align*}
\pgwt_{1-\delta_1-\delta_2,p}(X,Z)
&\leq \inf_{\set{\tilde{\pi} \in \mathcal{M}_+(X\times Y):\, \tilde{\pi}_X = \tilde{\mu}^c_{X},\, \tilde{\pi}_Y = \tilde{\mu}^c_{Y},\, \abs{\tilde{\pi}} = \delta_c }} \norm{d_X - d_Y}_{L^p(\tilde{\pi}\otimes \tilde{\pi})} \\
&\quad + \inf_{\set{\tilde{\pi} \in \mathcal{M}_+(Y\times Z):\, \tilde{\pi}_Y = \tilde{\mu}^c_{Y},\, \tilde{\pi}_Z = \tilde{\mu}^c_{Z},\, \abs{\tilde{\pi}} = \delta_c }} \norm{d_Y - d_Z}_{L^p(\tilde{\pi}\otimes \tilde{\pi})} \\
&\leq \norm{d_X - d_Y}_{L^p(\tilde{\pi}^c_1\otimes \tilde{\pi}^c_1)} + \norm{d_Y - d_Z}_{L^p(\tilde{\pi}^c_2\otimes \tilde{\pi}^c_2)} \\
&\leq \norm{d_X - d_Y}_{L^p(\tilde{\pi}_1\otimes \tilde{\pi}_1)} + \norm{d_Y - d_Z}_{L^p(\tilde{\pi}_2\otimes \tilde{\pi}_2)} \\
&\leq k\delta_1 + k\delta_2.
\end{align*}
Since $\delta_1$ and $\delta_2$ are arbitrary, we can conclude that
\[
\pgwt_{1-\delta,p}(X,Z) \leq k \delta,
\]
for any $\delta \in (\Bar{\delta}_1 + \Bar{\delta}_2, 1]$.
Since $\pgwt_p^k(X,Z)$ is the infimum $\delta \in [0,1]$ such that $\pgwt_{1-\delta,p}(X,Z) \leq k\delta$, we have $\pgwt_p^k(X,Z) \leq \Bar{\delta}_1 + \Bar{\delta}_2$. 
\end{proof}

\begin{corollary}[Metric properties for $\pgw_p^k$]\label{cor:pgwpkmetric}
    The robust partial Gromov-Wasserstein distance $\pgw_{p}^k$ defines a metric on the space of isomorphism classes of compact, fully supported pmm-spaces.
\end{corollary}

\begin{proof}
Non-degeneracy is already proven in Theorem \ref{thm:nondeg-isomorphism}. Symmetry is clear since we force $\eps_1 = \eps_2$. It remains to show the triangle inequality.

By Proposition~\ref{prop:equivalence}, we have 
\[
\pgwt_{1 - \delta,p}(X,Y) = \frac{1}{1+\eps} \pgw_{\eps,p}(X,Y),
\]
with $1+ \eps = \frac{1}{1 - \delta}$.
Thus,
\[
\inf\set{
\frac{\eps}{1+\eps}\geq 0 : \frac{1}{1+\eps} \pgw_{\eps,p}(X,Y) \leq k\frac{\eps}{1+\eps} 
}
= \inf\set{
0\leq \delta \leq1 : \pgwt_{1-\delta,p}(X,Y) \leq k\delta
},
\]
satisfies the triangle inequality.
\end{proof}

\paragraph{Incompleteness.} 
It was shown by example in \cite{memoli2014gw-overview} that the metric $\gw_p$ ($1\leq p < \infty$) on the space of isomorphism classes of compact pmm-spaces is incomplete. We re-use the example and show that $\pgw_p^k$ ($1\leq p < \infty$) is also incomplete. 

\begin{theorem}[Incompleteness of $\pgw_p^k$]\label{thm:incomplete}
    Let $1 \leq p < \infty$ and $k > 0$.
    The space of isomorphism classes of compact pmm-spaces endowed with the metric $\pgw_p^k$ is incomplete.
\end{theorem}

\begin{proof}
     For each $n \in \mathbb{N}$, let $\Delta_n$ denote the compact pmm-space consisting of an $n$-point set with pairwise distance one equipped with the uniform measure. We label $\Delta_n$ by $\set{x_i}_{i \in [n]}$.
    Let $m \in \mathbb{N}$. We label $\Delta_{nm}$ by $\set{y_{jk}}_{j \in [n], k \in [m]}$. It is clear that
    \begin{equation*}
        \pi(x_i, y_{jk}) := \frac{\delta_{ij}}{nm}
        \text{ for all } i,j \in [n] \text{ and } k \in [m]
    \end{equation*}
    defines an exact coupling of $\Delta_n$ and $\Delta_{nm}$. In particular, $\pi \in \C_\eps$ for all $\eps \geq 0$ (recall that $\eps$ is a scalar in this section). Recall the distortion function $\dis_p(\cdot)$ from Definition \ref{def:distortion}.
    By a straightforward calculation, we have
    \begin{equation*}
        \pgw_{\eps,p}(\Delta_n,\Delta_{nm}) \leq \dis_p(\pi) \leq n^{-1/p}.
    \end{equation*}
    As a consequence, 
    \begin{equation*}
        \pgw_p^k(\Delta_n,\Delta_{nm}) \leq \frac{k^{-1}n^{-1/p}}{1+k^{-1}n^{-1/p}}.
    \end{equation*}
    Similarly,
    \begin{equation*}
        \pgw_p^k(\Delta_m,\Delta_{nm}) \leq \frac{k^{-1}m^{-1/p}}{1+k^{-1}m^{-1/p}}.
    \end{equation*}
    By the triangle inequality for $\pgw_p^k$, we have
    \begin{equation*}
        \pgw_p^k(\Delta_n,\Delta_m) \leq \pgw_p^k(\Delta_n,\Delta_{nm}) + \pgw_p^k(\Delta_{nm},\Delta_m) \leq \frac{k^{-1}n^{-1/p}}{1+k^{-1}n^{-1/p}} + \frac{k^{-1}m^{-1/p}}{1+k^{-1}m^{-1/p}}.
        \label{eq:incomplete-1}
    \end{equation*}
    We see that $(\Delta_n)_{n\in\mathbb{N}}$ is a Cauchy sequence with respect to $\pgw_p^k$. However, a limiting object for this sequence cannot be compact, since it
    has to contain a countably infinite set of points with pairwise distance one.
\end{proof}

\subsection{Topological Equivalence}

The distance $\pgw_p^k$ defines a metric on the space of pmm-spaces, considered up to isomorphism, as does the Gromov-Wasserstein distance $\gw_p$. A natural question is whether they induce the same topology, and we establish that this is the case in the next result.

\begin{theorem}[Equivalence of Topologies]\label{thm:equivalence_of_topologies}
    Consider the topologies induced by $\pgw_p^k$ and $\gw_p$ on the space of isomorphism classes of compact pmm-spaces. If $p \in [1,\infty)$ then the topologies are equivalent. However, the topology induced by $\pgw_\infty^k$ is strictly coarser than that induced by $\gw_\infty$. 
\end{theorem}

\begin{proof}
    Since the topologies under consideration are metric topologies, it suffices to check sequential convergence. That is, we will show that, for any sequence of pmm-spaces $(X_n)_{n=1}^\infty$, 
    \begin{equation*}
        \gw_p(X_n,X) \to 0 \quad \mbox{implies} \quad \pgw_p^k(X_n,X) \to 0
    \end{equation*}
    holds in general, and that the converse holds when $p < \infty$.
    
    First assume that $\gw_p(X_n,X) \to 0$. For any $\eps \geq 0$, if $\gw_p(X_n,X) \leq k\eps$ then $\pgw_{\eps,p}(X_n,X) \leq k\eps$ (since $\pgw_{\eps,p} \leq \gw_p$, in general), so that 
    \[
    \pgw_p^k(X_n,X) = \inf\left\{\frac{\eps}{1+\eps} : \pgw_{\eps,p}(X_n,X) \leq k\eps\right\} \leq \inf\left\{\frac{\eps}{1+\eps} : \gw_{p}(X_n,X) \leq k\eps\right\} \to 0.
    \]
    
    To prove the converse, let $p < \infty$ and suppose that $\pgw_p^k(X_n,X) \to 0$. Then there exists a sequence $(\eps_n)_n$ of non-negative real numbers such that $\eps_n \to 0$ and $\pgw_{\eps_n,p}(X_n,X) \leq k \eps_n$. For each $n$, let $\pi_n \in \mathcal{C}_{\eps_n}(\mu_{X_n},\mu_X)$ be an optimal coupling (see Theorem \ref{thm:realization}), so that 
    \[
    \pgw_{\eps_n,p}(X_n,X) = \gw_p((X_n)_{\pi_n},X_{\pi_n}),
    \]
    using the notation and the equivalence \eqref{eq:thm:nondegen-2} of Theorem \ref{thm:approxi-nondegen}. Recall Definition \ref{def:ecc and rad} for the circumradius $\mathsf{rad}_p(X
    )$. 
    Applying the triangle inequality for $\gw_p$ and the estimates \eqref{eq:thm:nondegen-1}, we have 
    \begin{align*}
    \gw_p(X_n,X) &\leq \gw_p(X_n,(X_n)_{\pi_n}) + \gw_p((X_n)_{\pi_n},X_{\pi_n}) + \gw_p(X_{\pi_n},X) \\
    &\leq 2 \cdot 2^{1/q} \eps_n^{1/p} \mathsf{rad}_p(X_n) +  k \eps_n + 2 \cdot 2^{1/q} \eps_n^{1/p} \mathsf{rad}_p(X),
    \end{align*}
    so that $\gw_p(X_n,X)$ is bounded above by a sequence that tends to $0$ as $n \to \infty$. 

    Observe that the proof strategy for the converse fails when $p=\infty$, because the estimates \eqref{eq:thm:nondegen-1} are no longer effective for inducing convergence (see Remark \ref{rem:p_infinity_case}). We have already shown that the $\pgw_\infty^k$ topology is coarser than that of $\gw_\infty$. Strictness is demonstrated by Example \ref{ex:strictly_coarser} below.
\end{proof}

\begin{example}\label{ex:strictly_coarser}
    We will now construct a sequence $X_n$ of pmm-spaces which converge to a pmm-space $X$ in the $\pgw_\infty^k$ topology, but not in the $\gw_\infty$ topology. This completes the proof of the claim that the former topology is strictly coarser than that of the latter.
    
    Let $X$ denote the unique one-point pmm-space, whose point will be denoted $\star$. For each $n=2,3,\ldots$, let $X_n = (\{x,y\},d_n,\mu_n)$, where $d_n(x,y) \equiv 1$ (independent of $n$) and $\mu_n$ is defined by $\mu_n(x) = 1-\frac{1}{n}$ and $\mu_n(y) = \frac{1}{n}$. Then $\gw_\infty(X_n,X) = 1$ for all $n$, so that $X_n \not \to X$ in the $\gw_\infty$ topology. On the other hand, we claim that 
    \[
    \pgw_\infty^k(X_n,X) \leq \frac{1}{n},
    \]
    which will complete the proof. Indeed, setting $\eps = \frac{1}{n-1}$, we have
    \[
    \pgw_{\eps,\infty}(X_n,X) = 0 \leq k \eps.
    \]
    To see the equality, consider the measure $\pi$ on $X \times X_n$ defined by $\pi(\star,x) = 1$, $\pi(\star,y) = 0$. This measure belongs to $\mathcal{C}_\eps(\mu_X,\mu_n)$: obviously, $\pi_X = \mu_X$, while 
    \[
    \pi_Y(x) = 1 = \left(1+\frac{1}{n-1}\right)\left(1-\frac{1}{n}\right) = (1+\eps)\mu_n(x) \quad \mbox{and} \quad \pi_Y(y) = 0 \leq (1+\eps)\mu_Y(y).
    \]
    Moreover, $\dis_\infty(\pi) = 0$. This implies that
    \[
    \pgw_\infty^k(X_n,X) \leq \frac{\eps}{1+\eps} = \frac{\frac{1}{n-1}}{1+ \frac{1}{n-1}} = \frac{1}{n}.
    \]
\end{example}

\subsection{Robustness}
In this section, we investigate the robustness properties of our robust partial GW distance. Our robustness result factors through a more general result on the behavior of arbitrary parametric families of kernels, after being fed through a $\pgw_p^k$-type construction. We state this general result after proving a technical lemma.

\begin{lemma}\label{lem:posetlemma}
    Let $(\mathcal{E},<)$ be a partially ordered set. Suppose there is a weakly decreasing function $f:\mathcal{E}\to \mathcal{E}$ and a weakly increasing function $g:\mathcal{E}\to \mathcal{E}$ with $f(a)= g(b)$ for some $a,b\in \mathcal{E}$ such that $a\wedge b,a\vee b\in\mathcal{E}$, i.e., both the meet and join of $a$ and $b$ lie in the poset $\mathcal{E}$. Let $h:\mathcal{E}\to \mathcal{E}$ be a strictly increasing function, define the set $\mathcal{S} = \{h(\eps)\in \mathcal{E}: f(\eps)\leq g(\eps)\}$, and suppose $s = \inf \mathcal{S}\in \mathcal{S}$. Then, $s  \leq a\vee b$. Furthermore, if $g$ is strictly increasing, then $a\wedge b\leq s$.
\end{lemma}

\begin{proof}
    Consider the case where $a\leq b$. In this case, we have $f(b)\leq f(a) = g(b)$. Hence, $b\in \mathcal{S}$, so $s \leq b = a\vee b$. Next, if $g$ is strictly increasing, then assuming $s < a$ leads to a contradiction:
    \begin{equation*}
         g(s) < g(a) \leq g(b) = f(a) \leq f(s) \leq g(s).
    \end{equation*}
    Thus, we have $s \geq a = a\wedge b$.
    
    Now, consider the case where $b< a$. Here we have $f(a) = g(b) \leq g(a)$, so $a\in \mathcal{S}$, and $s\leq a = a\vee b$. As in the previous case, for $g$ strictly increasing, if we assume $s < b$, then we reach a contradiction:
    \begin{equation*}
        g(s) < g(b) = f(a) \leq f(b) \leq f(s) \leq g(s),
    \end{equation*}
    so $s \geq b = a\wedge b$. This proves the lemma.
\end{proof}

We now give our main result for this subsection. 

\begin{theorem}\label{thm:simfamily}
    Let $\mathcal{W}$ be a set. 
    Suppose $d_\eps:\mathcal{W}\times\mathcal{W}\to\R$ is a family of symmetric functions parametrized by $\eps\geq 0$ such that $\eps\mapsto d_\eps(x,y)$ is right continuous and weakly decreasing for any fixed pair $x,y\in\mathcal{W}$. Define the family of functions $d^k(x,y) = \inf\{\frac{\eps}{1+\eps}: d_\eps(x,y)\leq k\eps \}$, for $k> 0$. Suppose there exists some $\eps>0$ such that $d_\eps(y,\tilde y)=0$. If $d^k$ satisfies the triangle inequality, then
    \begin{equation*}
        \left|d^k(x,y)-d^k(x,\tilde y)\right|\leq \eps.
    \end{equation*}
\end{theorem}

\begin{proof}
    By Lemma \ref{lem:posetlemma}, we have $d^k(y,\tilde y)\leq \eps$. By the triangle inequality, we have
    \begin{equation*}
        d^k(x,y)\leq d^k(x,\tilde y)+d^k(\tilde y,y) \leq d^k(x,\tilde y) + \eps,
    \end{equation*}
    and
    \begin{equation*}
        d^k(x,\tilde y)\leq d^k(x,y)+d^k(y,\tilde y) \leq d^k(x,y) + \eps
    \end{equation*}
    as desired.
\end{proof}

\begin{corollary}[Robustness of $\pgw_p^k$]
\label{cor:robust}
    Let $k>0$, $p \in [1,\infty]$. Let $X,Y,\tilde{Y}$ be compact pmm-spaces. Suppose that $\pgw_{\eps,p}(Y,\tilde Y) = 0$ for some $\eps > 0$. Then
    \[
        \left|\pgw_p^k(X,Y)-\pgw_p^k(X,\tilde Y)\right|\leq \eps.
    \]
\end{corollary}

\begin{proof} 
    Recall from the beginning of Section \ref{sec:robust_pGW} (Page \pageref{sec:robust_pGW}) that $\eps = (\eps_1,\eps_2)$ with $\eps_1 = \eps_2$.
    Thanks to Theorem \ref{thm:monotone-convergence}, $\pgw_{\eps,p}$ satisfies the assumptions on the family $d_\eps$ from Theorem \ref{thm:simfamily}. The result follows from Theorem \ref{thm:simfamily}.
\end{proof}

We restate Corollary \ref{cor:robust} in the language of Huber’s contamination model popularized in robust statistics \cite{chen2018contamination}. Let $\delta > 0$ and $L \geq 1$.
Given a pmm-space $X = (X, d_X, \mu_X)$, we call $X^c = (X, d_X, \mu_X^c)$ a \textbf{$(\delta,L)$-contamination of $X$} if
\begin{equation}
    \mu_X^c = (1-\delta)\mu_X + \delta\mu^c
    \label{eq:contaminate-1}
\end{equation}
for some outlier distribution $\mu^c \in \P(X)$ satisfying
\begin{equation}
    \mu^c \leq L\cdot \mu_X.
    \label{eq:contaminate-2}
\end{equation}
We see from \eqref{eq:contaminate-1} and \eqref{eq:contaminate-2} that if $X^c$ is a $(\delta,L)$-contamination of $X$, then 
\begin{equation}
    \left\|\frac{d\mu_X^c}{d\mu_X} \right\|_{L^\infty(\mu_X)} \leq (1-\delta) + \delta L.
    \label{eq:contaminate-3}
\end{equation}

\begin{corollary}[Robustness of $\pgw_p^k$ for contaminated models]\label{cor:robust-contamination}
    Let $k > 0$, let $\eps,\delta,L > 0$ satisfy 
    \begin{equation}
        \delta(L-1) \leq \eps,
        \label{eq:contaminate-4}
    \end{equation}
    and let $X_0$ and $X$ be compact pmm-spaces. 
    If $X^c$ is a $(\delta,L)$-contamination of $X$ then, for all $p \in [1,\infty]$,
    \begin{equation}
        \abs{\pgw_p^k(X_0, X) - \pgw_p^k(X_0,X^c)} \leq \eps.
        \label{eq:contaminate-5}
    \end{equation}
\end{corollary}

\begin{proof}
    In view of \eqref{eq:contaminate-3} and \eqref{eq:contaminate-4}, we have $\norm{\frac{d\mu_X^c}{d\mu_X}}_{L^\infty(\mu_X)} \leq 1+\eps$. Using the same notation as in Theorem \ref{thm:approxi-nondegen}, we have $X^c \sim_\eps X$ and $\pgw_{\eps,p}(X^c,X) = 0$, so that \eqref{eq:contaminate-5} follows from Corollary \ref{cor:robust}.
\end{proof}

\begin{remark}
    The derived bound \eqref{eq:contaminate-5} suggests that the Gromov-Wasserstein metric $\pgw_p^k$ provides provable robustness against contamination in the Huber model, hence the prefix ``robust'' in the name. Our result is comparable to \cite[Theorem 3.1]{RSZ2024}, where the control \eqref{eq:contaminate-4} on $ (\delta,L)$-contamination resembles Assumption (A1) therein. We also note that our notion of robustness differs from that in \cite{kong2024outlier} (see Section \ref{sect: relation to kong}), and our counterpart for \cite[Theorem 2.3]{kong2024outlier} is Theorem \ref{thm:approxi-nondegen}.
\end{remark}

 It is worth pointing out that the classical Gromov-Wasserstein distance does not enjoy a similar robustness property. Precisely, we have the following result.

\begin{proposition}
    Let $\epsilon, \delta > 0$ and $L > 1$ such that $\delta(L-1) \leq \epsilon$, and let $p \in [1,\infty]$. For any $C > 0$, there exist compact pmm-spaces $X_0$ and $X$  and a $(\delta,L)$-contamination $X^c$ of $X$ such that 
    \[
    |\gw_p(X_0,X) - \gw_p(X_0,X^c)| > C.
    \]
\end{proposition}

\begin{proof}
    Let us assume for the sake of simplicity that $L \geq 2$ and $\delta < 1$ (the remaining cases can be handled by a similar, but more involved, construction). Let $X_0$ be the one-point pmm space and let $(X,d_X,\mu_X)$ be a two-point pmm space with $X = \{x,y\}$, $\mu_X$ uniform, and $d_X(x,y) \coloneqq d$, where the value of $d$ is to be determined. Let $\mu^c$ be a Dirac measure on $x$ (i.e., $\mu^c(y) = 0$)---since we assumed $L \geq 2$, $\mu^c$ satisfies \eqref{eq:contaminate-2}. Then the measure $\mu_X^c$ on $X^c$ is as defined in \eqref{eq:contaminate-1}; explicitly, 
    \[
    \mu_X^c(x) = \frac{1}{2}(1+\delta), \qquad \mu_X^c(y) = \frac{1}{2}(1-\delta).
    \]
    By \cite[Theorem 5.1]{memoli2011-gw}, we have that 
    \[
    \gw_p(X_0,X) = \mathrm{diam}_p(X) \quad \mbox{and} \quad \gw_p(X_0,X^c) = \mathrm{diam}_p(X^c),
    \]
    where the \define{$p$-diameter} of $X$ is given by 
    \[
    \mathrm{diam}_p(X) = \|d_X\|_{L^p(\mu_X \otimes \mu_X)}.
    \]
    A simple computation shows that the $p$-diameters are given by 
    \[
    \mathrm{diam}_p(X) = \frac{d}{2^{1/p}} \quad \mbox{and} \quad \mathrm{diam}_p(X^c) = \frac{d}{2^{1/p}} (1-\delta^2)^{1/p},
    \]
    so that 
    \[
    |\gw_p(X_0,X) - \gw_p(X_0,X^c)| = \frac{d}{2^{1/p}}\left(1-(1-\delta^2)^{1/p}\right). 
    \]
    Since $1-(1-\delta^2)^{1/p}$ is strictly positive, this quantity can be made arbitrarily large by an appropriate choice of $d$. 
\end{proof}

\subsection{Numerical Implementation and Examples}\label{sec:numerics}

While the focus of this article is on theoretical properties of the robust partial GW distances introduced herein, we provide in this subsection a basic numerical implementation of the $\mathsf{mPGW_2^k}$ metric and some proof-of-concept examples for the sake of providing intuitive illustrations of its properties. More serious explorations of algorithmic aspects of the distance and its potential applications will be the subject of follow-up work.

\paragraph{Implementation.} We implemented the $\mathsf{mPGW}_2^k$ version of our metric (which is equivalent to $\mathsf{PGW}_2^k$) by searching over $\delta$ values to find a zero of the function $f_k:[0,1] \to \R$ defined by
\[
f_k(\delta) = \mathsf{mPGW}_{1-\delta,2}(X,Y) - k \delta.
\]
The function $f_k$ is approximated via an implementation of the Chapel et al.\ distance  $\mathsf{mPGW}_{1-\delta,2}$ provided in the \texttt{Python Optimal Transport} software package \cite{flamary2021pot}. This implementation uses conjugate gradient descent to approximate the solution of the optimization problem \eqref{eq:chapel_PGW}. Assuming the spaces being compared have size $\leq n$, each gradient step has computational cost bounded by $O(n^3 \log(n))$ \cite{peyre2016gromov}, leading to an overall complexity of $O(m \cdot n^3 \log(n))$ to approximate $f_k$, where $m$ is the predetermined max number of iterations allowed in the gradient descent. We find a zero of $f_k$ via a simple bisection method, which finds a zero up to a fixed tolerance $t > 0$ in $\log_2(1/t)$ iterations. The overall complexity of approximating  $\mathsf{PGW}_2^k(X,Y)$ is therefore $O(m \cdot n^3 \log(n) \log_2(1/t))$. 

\paragraph{Examples.} Our synthetic examples are primarily focused on illustrating the intuitive behavior of the $\mathsf{mPGW}_2^k$ metric. Our experiments are summarized in Figures \ref{fig:circles_and_ellipses}---\ref{fig:matching_comparison}, and details are provided below.

\begin{figure}
    \centering
    \includegraphics[width=0.9\textwidth]{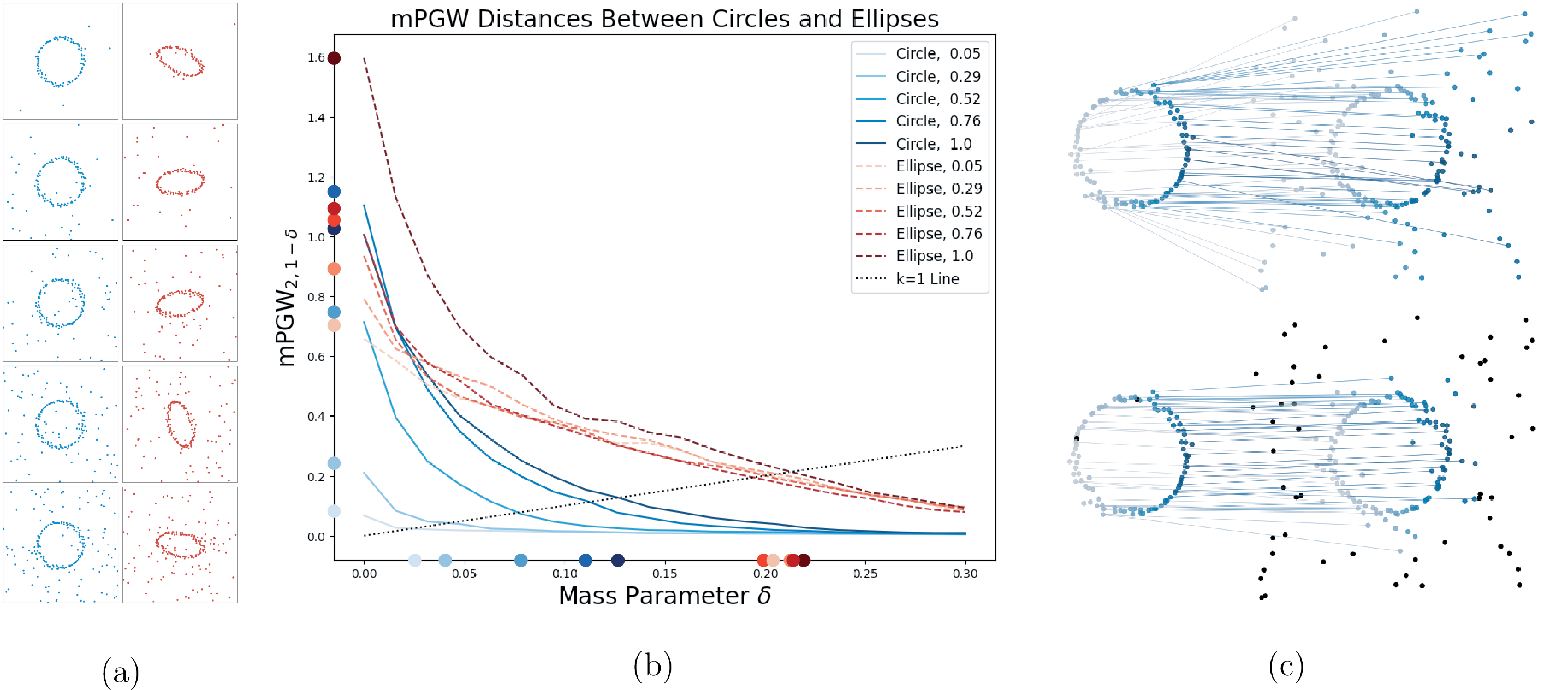}
    \caption{Illustration of $\mathsf{mPGW}_p^k$ on synthetic data (see Example \ref{ex:circles_and_ellipses} for details). {\bf (a)} Synthetic dataset of noisy circles and ellipses. {\bf(b)} Plots of $\mathsf{mPGW}_{1-\delta,2}(X_0,X)$ between the template shape $X_0$ and shapes $X$ in the dataset, with $\delta \in [0,0.3]$, and the line $y = \delta$. The values at zero are equal to $\mathsf{GW}_2(X_0,X)$, which we emphasize with colored points along the $y$-axis. The $\delta$-values where the curves intersect the $y = \delta$ line give $\mathsf{mPGW}_2^1(X_0,X)$, as we emphasize with colored points along the $x$-axis. Observe that the points on the $x$-axis give a more intuitive clustering pattern than those along the $y$-axis. {\bf (c)} Visualization of the matchings induced by Gromov-Wasserstein distance (top) and robust partial Gromov-Wasserstein distance (bottom).}
    \label{fig:circles_and_ellipses}
\end{figure}

\begin{example}[Circles and Ellipses, Figure \ref{fig:circles_and_ellipses}]\label{ex:circles_and_ellipses}
    In our first example, we work with a synthetic dataset consisting of noisy circles and ellipses in the plane. The circle dataset consists of five shapes: each consists of 100 points evenly spaced on a unit circle, to which random Gaussian perturbations are applied, together with $n$ points sampled uniformly at random from the square $[-5,5]^2$, where $n = 5, 28, 52, 76, 100$, respectively. We similarly construct randomly perturbed and rotated ellipses, each with semi-major axis  2 and semi-minor axis 1, with varying levels of uniform noise---see Figure \ref{fig:circles_and_ellipses}(a). Each shape $X$ is treated as pmm spaces by endowing it with Euclidean distance $d_X$ and a \textit{density-based probability measure} $\mu_X$, defined as follows. For a predetermined radius $r$ (we use $r = 2$ in our examples), for each point $x \in X$, $\mu_X$ is the normalization of the (non-probability) measure $\hat{\mu}_X$ with 
    \begin{equation}\label{eqn:density_based_measure}
    \hat{\mu}_X(x) = |\{x' \in X \mid d_X(x,x') \leq r\}|/|X|,
    \end{equation}
    where $|\cdot|$ is used to denote the cardinality of each set.

    We provide intuition for the behavior of $\mathsf{mPGW}_2^k$ (with $k=1$) as follows. We construct another circle with random perturbations, but with no additive uniform noise and refer to this as the \textit{template shape}, $X_0$. For each shape $X$ in our dataset of noisy circles and ellipses, we compute $\mathsf{mPGW}_{1-\delta,2}(X_0,X)$ for $\delta \in [0,0.3]$ and plot the resulting curves in Figure \ref{fig:circles_and_ellipses}(b). When $\delta = 0$, this corresponds to $\gw_2(X_0,X)$; this is emphasized in the plot by drawing colored dots along the $y$-axis corresponding to these distances. We also plot the line $y = k\delta$ over this range. The $\delta$-value of intersection of the curve  $\mathsf{mPGW}_{1-\delta,2}(X_0,X)$ with the line $y = k\delta$ is equal to $\mathsf{mPGW}_2^k(X_0,X)$; this is emphasized in the plot by drawing colored dots along the $x$-axis corresponding to these intersections. Observe that the GW distances along the $y$-axis mix the circle and ellipse classes and do not display a strong intuitive pattern. On the other hand, the dots along the $x$-axis---corresponding to our robust partial GW distance---cluster the circles and the ellipses in an intuitive manner. 

    The source of the robustness of our metric is intuitively explained by Figure \ref{fig:circles_and_ellipses}(c). We first show the matching between the template shape $X_0$ and one of the noisy circles $X$ induced by GW distance. Specifically, we begin with a GW optimal coupling $\pi$ (i.e., a coupling realizing $\mathsf{GW}_2(X_0,X)$) and infer a matching by declaring point $x_i$ in $X_0$ to be \textit{matched} to a point $x_j$ in $X$ if $i = \max_{i} \pi(i,j)$. This matching is illustrated in two ways: (1) by coloring the points of $X_0$ and transferring that coloring to $X$ via the matching, and (2) by drawing lines between a subset (40\%) of the matched points. The plot illustrates the fact that the GW matching forces the noise points to be matched. On the other hand, the same visualization method is used to show the matching induced by $\mathsf{mPGW}_2^k$. Here, points colored black are not involved in any matching; we see that the robust partial GW distance automatically filters out the majority of the noise points from the matching.  
\end{example}

\begin{figure}
    \centering
    \includegraphics[width=0.95\textwidth]{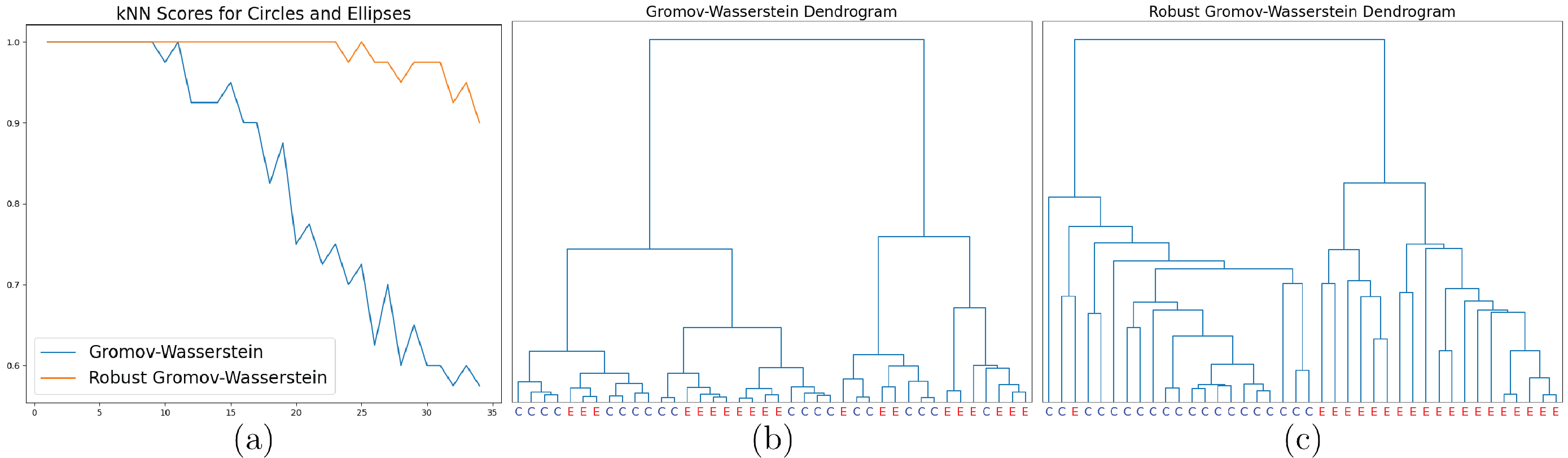}
    \caption{Quantitative evaluation of $\mathsf{mPGW}_p^k$ on synthetic data (see Example \ref{ex:clustering_circles_ellipses} for details). {\bf (a)} $K$-Nearest Neighbor scores for Gromov-Wasserstein $\mathsf{GW}_2$ and partial robust Gromov-Wasserstein $\mathsf{mPGW}_2^1$ distances on a synthetic dataset of noisy circles and ellipses. {\bf (b)} Dendrogram of the Gromov-Wasserstein distance matrix, with respect to the Ward linkage.  {\bf (c)} Dendrogram of the robust partial Gromov-Wasserstein distance matrix, with respect to the Ward linkage. Observe that $\mathsf{mPGW}_2^1$ gives a more consistent clustering pattern with respect to shape class than $\mathsf{GW}_2$.}
    \label{fig:clustering}
\end{figure}

\begin{example}[Clustering Circles and Ellipses, Figure \ref{fig:clustering}]\label{ex:clustering_circles_ellipses}
    This example gives a more quantitative picture of the behavior of $\mathsf{PGW}_2^1$ by extending the synthetic dataset considered in Example \ref{ex:circles_and_ellipses}. Here, we construct a similar dataset of noisy circles and ellipses, with 20 shapes from each class. All settings are the same, except the number of points in the additive uniform noise runs from 5 to 300. We compute the pairwise distance matrices for the dataset of 40 shapes with respect to $\gw_2$ and $\mathsf{mPGW}_2^1$, and quantify the clustering properties with respect to the shape classes in two ways.

    For each distance matrix, we compute a $K$-Nearest Neighbors classification score, with $K = 1,2,\ldots,35$ as follows: for a fixed $K$, the classification score is the percentage of shapes in the dataset with the property that more than $K/2$ of its nearest neighbors belong to the same class. The scores for each metric are plotted in Figure \ref{fig:clustering}(a). We see that both metrics perform well for small values of $K$, but that the GW distance scores degrade more rapidly as $K$ grows than those of the robust partial GW distance, indicating that there is stronger class clustering with respect to the latter metric. 

    We further illustrate the clustering properties of the metrics by displaying dendrograms for each matrix. The dendrograms are computed via a \textit{Ward} linkage. Leaves of each dendrogram are labeled by class (blue `C' for circles and red `E' for ellipses). The clustering structure for GW distance revealed by the dendrogram in Figure \ref{fig:clustering}(b) shows that the two classes are relatively mixed. On the other hand, the robust partial GW dendrogram, shown in Figure \ref{fig:clustering}(c), illustrates that the classes are almost perfectly clustered.
\end{example}

\begin{figure}[htbp]
    \centering

    \begin{subfigure}[b]{0.45\textwidth}
        \centering
        \fbox{\includegraphics[width=\textwidth]{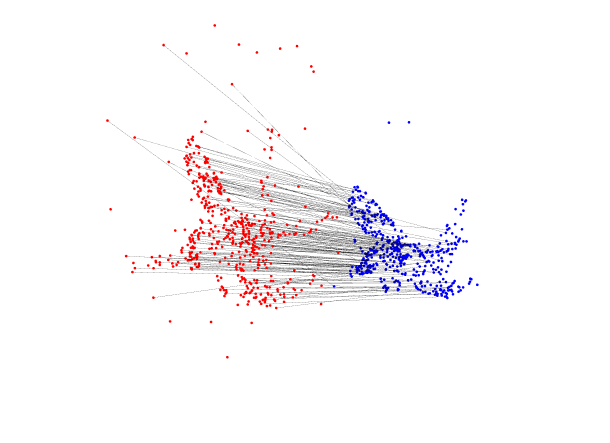}}
        \caption{Airplane - GW matching}
        \label{fig:airplane_pgw}
    \end{subfigure}
    \hfill
    \begin{subfigure}[b]{0.45\textwidth}
        \centering
        \fbox{\includegraphics[width=\textwidth]{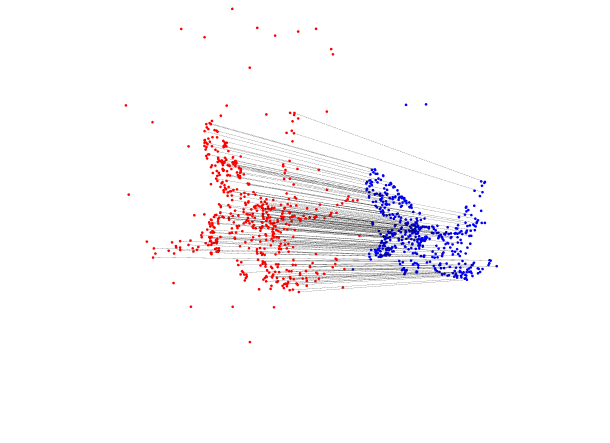}} 
        \caption{Airplane - mPGW matching}
        \label{fig:airplane_mpwg}
    \end{subfigure}
    \vspace{0.5cm}

    \begin{subfigure}[b]{0.45\textwidth}
        \centering
        \fbox{\includegraphics[width=\textwidth]{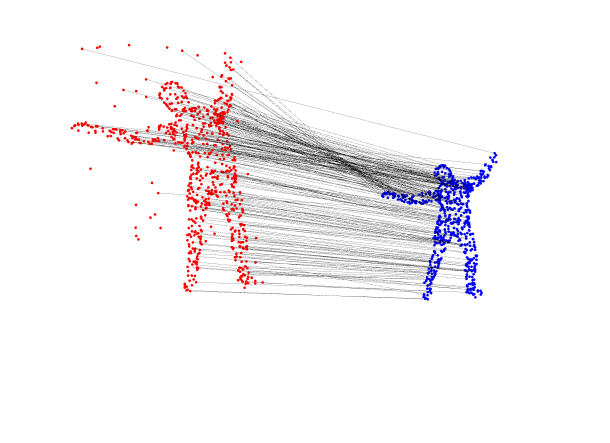}}
        \caption{Person - GW matching}
        \label{fig:airplane_pgw}
    \end{subfigure}
    \hfill
    \begin{subfigure}[b]{0.45\textwidth}
        \centering
        \fbox{\includegraphics[width=\textwidth]{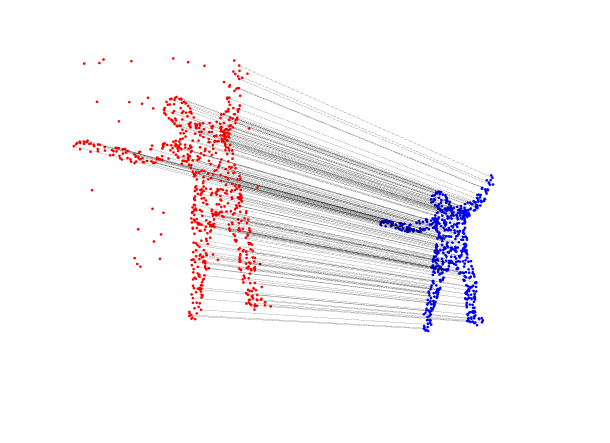}}
        \caption{Person - mPGW matching}
        \label{fig:airplane_mpwg}
    \end{subfigure}
    \vspace{0.5cm}

    \begin{subfigure}[b]{0.45\textwidth}
        \centering
        \fbox{\includegraphics[width=\textwidth]{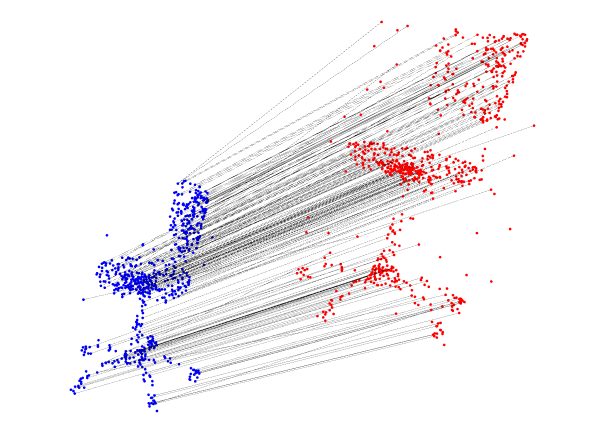}}
        \caption{Chair - GW matching}
        \label{fig:airplane_pgw}
    \end{subfigure}
    \hfill
    \begin{subfigure}[b]{0.45\textwidth}
        \centering
        \fbox{\includegraphics[width=\textwidth]{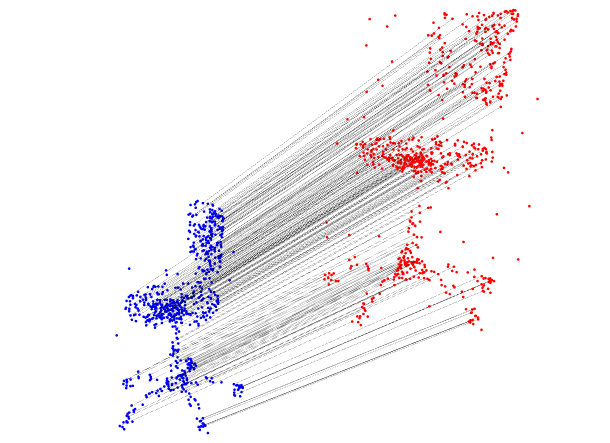}} 
        \caption{Chair - mPGW matching}
        \label{fig:airplane_mpwg}
    \end{subfigure}
    \vspace{0.5cm}

    \caption{Visual comparisons of the matching results for the airplane, person, and chair, using GW and mPGW.}
    \label{fig:matching_comparison}
\end{figure}

\begin{example}[3D Matching Comparison, Figure \ref{fig:matching_comparison}]\label{ex:matching_comparison} 
In this experiment, we used 3D mesh data from the ModelNet40 dataset, which contains high-resolution meshes of various objects\cite{wu2015-modelnet}. These meshes are then converted into 3D point clouds by sampling points from their surfaces. Although the data set includes a variety of objects, we focus on three representative objects: an airplane, a person, and a chair. These objects are represented as 3D point clouds, and we compare the matching results between the Gromov-Wasserstein distance metric and the robust partial Gromov-Wasserstein distance metric. These metrics are used to match noisy 3D point clouds to a template shape.

For each object, noisy point clouds are generated by introducing random perturbations and shifts to the original point cloud. The goal is to assess how each metric performs in terms of matching noisy data to the template shape. The matching process is visualized by connecting corresponding points between the template and the noisy shape. A subset of these connections is highlighted for clarity, with the lines displayed controlled by the line percent parameter, set to $30\%$ of the total matched pairs.

The visualizations are arranged such that the results from GW distance are shown first, followed by the results from mPGW, demonstrating how mPGW improves robustness by better handling noise and focusing on the essential structure of the object. It should be noted that the 3D results are presented as 2D screenshots taken from different perspectives.

In the first row of Figure \ref{fig:matching_comparison}, the results for the airplane are shown: The first plot (Figure \ref{fig:matching_comparison}(a)) illustrates the optimal matching using the GW distance, with noisy points shown in red and template points in blue. As expected, GW-based matching tends to match noise points more frequently, leading to a less robust result with more connections, some of which may be due to noise. In the second plot (Figure \ref{fig:matching_comparison}(b), we repeat the matching process using the mPGW distance, where the matching is more robust, excluding noise points and highlighting only the most relevant correspondences.

In the second row, we show the results for the person object. The matching process follows the same structure, first showing GW (Figure \ref{fig:matching_comparison}(c)), where GW not only struggles with noise but also incorrectly matches parts of the object, such as the hands, leading to erroneous correspondences. In the second plot (Figure \ref{fig:matching_comparison}(d)), we repeat the matching process using the mPGW distance, where the matching is more robust, excluding both noise points and incorrect correspondences, focusing on the most relevant correspondences.

Finally, in the third row, the results for the chair object are displayed, first with GW (Figure \ref{fig:matching_comparison}(e)) and then with mPGW (Figure \ref{fig:matching_comparison}(f)).

The comparison emphasizes the advantages of the mPGW metric in filtering out noise and providing a more accurate matching, especially when dealing with noisy 3D point clouds. Cleaner and more meaningful matches in the mPGW plots demonstrate the robustness of the metric to preserve the essential shape structure.

\end{example}

\section*{Acknowledgements} This work was initiated during the AMS Mathematical Research Communities (MRC) Workshop on \textit{Mathematics of Adversarial, Interpretable, and Explainable AI}, in June 2024. We would like to thank the AMS,  and especially the staff who helped organize and run the workshop, for providing a very fruitful working environment. The AMS MRC program was funded by NSF grant DMS 1916439. In addition, Needham was partially supported by NSF grants DMS 2107808 and 2324962. Li was partially supported by NSF grant DMS 2410140.

\bibliographystyle{plain}
\bibliography{pgw}

\end{document}